\documentclass[11pt]{amsart}

\usepackage{amsmath,amssymb,amsthm,tikz, nicefrac, mathrsfs,upgreek, accents}
\usepackage{mathtools}
\usepackage{enumitem}

\newcommand{\seqnum}[1]{\href{http://oeis.org/#1}{\underline{#1}}}

\RequirePackage{color}
\RequirePackage[colorlinks, urlcolor=my-blue,linkcolor=my-blue,citecolor=my-blue]{hyperref}
\definecolor{my-blue}{rgb}{0.0,0.0,0.6}
\definecolor{my-purple}{RGB}{9,0,255}
\definecolor{my-red}{rgb}{0.5,0.0,0.0}
\definecolor{my-green}{rgb}{0.0,0.5,0.0}
\definecolor{nicos-red}{rgb}{0.75,0.0,0.0}
\definecolor{nicos-green}{rgb}{0.0,0.75,0.0}
\definecolor{light-gray}{gray}{0.6}
\definecolor{really-light-gray}{gray}{0.8}
\definecolor{sussexg}{rgb}{0.0,0.5,0.5}
\definecolor{sussexp}{rgb}{0.5,0.0,0.5}
\definecolor{purple(x11)}{rgb}{0.63, 0.36, 0.94}

\usepackage{dsfont} 
\newcommand{\idPF}{\mathrm{PF}_n(\mathrm{id})}

\newcommand{\id}{\mathrm{id}}

\usepackage[utf8]{inputenc}

\usepackage{nicefrac}

\usepackage{multicol}
\PassOptionsToPackage{dvipsnames}{xcolor}
\usepackage{tikz,xcolor}

\usetikzlibrary{shapes,arrows}
\usetikzlibrary{hobby}
\usetikzlibrary{decorations.markings}

\newtheorem{theorem}{\color{my-blue}{\sc Theorem}}[section]
\newtheorem{lemma}[theorem]{\color{my-blue} \sc Lemma}
\newtheorem{proposition}[theorem]{\color{my-blue} \sc Proposition}
\newtheorem{corollary}[theorem]{\color{my-blue} \sc Corollary}

\newtheorem{question}[theorem]{\color{my-blue} \sc Question}

\newtheorem{definition}[theorem]{\color{my-blue}  \sc Definition}

\numberwithin{equation}{section}
\theoremstyle{remark}
\newtheorem{remark}[theorem]{\color{my-blue} Remark}

\newcommand{\be}{\begin{equation}}
\newcommand{\ee}{\end{equation}}

\newcommand{\eqpd}{\, .}
\newcommand{\eqcom}{\, ,}
\newcommand{\stp}{\sqrt{2\pi}}

\newcommand*{\PF}{\mathsf{PF}}
\newcommand*{\NDPF}{\mathsf{PF}^{\uparrow}}

\newcommand{\x}{\textbf{x}}

\newcommand{\U}{\mathcal{U}}

\def\bE{\mathbb{E}}
\def\bN{\mathbb{N}}

\def\bZ{\mathbb{Z}}

 \def\Z{\bZ}

\def\N{\bN}

\usepackage{stackengine}

\DeclareMathOperator{\Var}{Var}

\def\E{\bE}

\definecolor{partcolor1}{rgb}{0.0,0.5,0.0}
\definecolor{partcolor2}{rgb}{0.0,0.5,0.0}

\definecolor{darkgreen}{rgb}{0.0,0.5,0.0}
\definecolor{darkblue}{rgb}{0.5,0.1,0.5}
\definecolor{deepblue}{rgb}{0.25,0.41,0.88}
\definecolor{nicosred}{rgb}{0.65,0.1,0.1}
\definecolor{light-gray}{gray}{0.7}
\allowdisplaybreaks[1]

\RequirePackage{datetime} 
\allowdisplaybreaks

\begin{document}
\usdate
\title[A Probabilistic Parking Process
and Labeled
IDLA]
{
A Probabilistic Parking Process
and Labeled
IDLA}

\author[Harris]{Pamela E. Harris}
\address[P. E.~Harris]{Department of Mathematical Sciences, University of Wisconsin, Milwaukee}
\email{\textcolor{blue}{\href{mailto:peharris@uwm.edu}{peharris@uwm.edu}}}

\author[Holleben]{Thiago Holleben}
\address[T.~Holleben]{Department of Mathematics \& Statistics,
Dalhousie University}
\email{\textcolor{blue}{\href{mailto:hollebenthiago@dal.ca}{hollebenthiago@dal.ca}}}
\author[Mart\'inez Mori]{J. Carlos Mart\'inez Mori}
\address[J.~C. Mart\'inez Mori]{H. Milton Stewart School of Industrial and Systems Engineering, Georgia Institute of Technology}
\email{\textcolor{blue}{\href{mailto:jcmm@gatech.edu}{jcmm@gatech.edu}}}

\author[Priestley]{Amanda Priestley}
\address[A.~Priestley]{Department of Computer Science, The University of Texas at Austin}
\email{\textcolor{blue}{\href{mailto:amandapriestley@utexas.edu}{amandapriestley@utexas.edu}}}

\author[Sullivan]{Keith Sullivan}
\address[K.~Sullivan]{Department of Mathematics and Statistics, University of Vermont}
\email{\textcolor{blue}{\href{mailto:Keith.Sullivan@uvm.edu}{Keith.Sullivan@uvm.edu}}}

\author[Wagenius]{Per Wagenius}
\address[P.~Wagenius]{Department of Mathematics and Statistics, University of Vermont}
\email{\textcolor{blue}{\href{mailto:Per.Wagenius@uvm.edu}{Per.Wagenius@uvm.edu}}}

\date{\today}

\begin{abstract}
In 1966, Konheim and Weiss \cite{konheimWeiss} introduced a now classical parking protocol. The deterministic process and its resultant objects, known as \textit{parking functions}, have since become a favorite object of study in enumerative combinatorics. In our work, we introduce and study a probabilistic variant of the classical parking protocol, which is closely related to Internal Diffusion Limited
Aggregation, or IDLA, introduced in 1991 by Diaconis and Fulton \cite{diaconis1991growth}. In particular, we compute the stationary distribution of this process when initiated with a particular class of initial preferences, of which weakly increasing parking functions are a subset. Furthermore, we compute the expected time it takes for the protocol to complete assuming all of the cars park, and prove that, in some cases, the parking process is negatively correlated.
In addition, we study statistics of uniformly random weakly increasing parking functions such as the distribution of 
the last entry, the probability that a specific set of cars is lucky, and the expected number of lucky cars.
\end{abstract}
\maketitle

\section{Introduction}
\label{sec: introduction}

In 1966 Konheim and Weiss \cite{konheimWeiss} introduced a now classical parking problem which is defined as follows. 
Fix $n \in \N= \{1, 2, 3, \ldots\}$ and consider a sequence of $n$ cars, each given labels in $[n] \coloneqq \{1,2,\dots,n\}$, attempting to park on the sites of the line segment $[n]$ according to the following protocol.
Initially, each car records its preferred parking spot in the preference list $\alpha \coloneqq (\alpha_1,\alpha_2, \dots, \alpha_n) \in [n]^n$, with $\alpha_i$ being the preference of car $i$.
The cars then enter the street from left to right, and attempt to park in increasing order of their labels. 
If their preferred parking spot is unoccupied, they are able to immediately park in their desired spot.
Otherwise, each car continues to drive to the next available spot to the right and parks there. 
If there is no such parking spot available, the car exits the street and is unable to park.
The set of preference lists $\alpha$ such that all cars can park under the aforementioned protocol are called \textit{parking functions}, and the set of all parking functions with $n$ cars is denoted by $\PF_n$. 
To give a few simple but illustrative examples, the preference list $(1,1,1, \dots, 1)$ is a parking function, as is any permutation of $[n]$. 
We let $S_n$ denote the set of permutations of $[n]$, and we write permutations in one-line notation. 
However, the preference list $(n,n,n, \dots, n)$ is not a parking function, as the second car already fails to park. 
Konheim and Weiss, interested in studying hash functions,  initiated the study of these combinatorial objects by proving that $|\PF_n| = (n+1)^{n-1}$.
Since then, classical parking functions and several new variants \cite{colaric2020interval, countingKnaples, fang2024vacillatingparkingfunctions} have been of great interest to the enumerative combinatorics community.

In addition, parking functions, along with their restrictions and generalizations, have connections to a vast range of mathematical fields. Past work connects parking functions to computing volumes of polytopes, hyperplane arrangements, representation theory, sorting algorithms, enumerating Boolean intervals in the weak Bruhat order of the symmetric group, facet structure of the permutahedron, and even to the continued fraction expansion of $\sqrt{2}$ and to the Tower of Hanoi game \cite{Aguillon2022OnPF,MR3724106,Harris2023LuckyCA,unit_perm,garcia2024defectiveparkingfunctionsyoung,fang2024vacillatingparkingfunctions,elder2025parking,MR1902680,HicksDiagonalHarmonics,MR3551576}. 
For a brief introduction to parking functions and some of their properties, we recommend \cite{mori2024parkingfunction}. For a detailed overview of known results about parking functions, we point the reader to Yan \cite{yan2015parking}, and for many open problems involving parking functions we recommend \cite{carlson2020parking}. 

In this work, we will be interested in discussing the behavior of a probabilistic parking protocol when initiated with a preference vector which is a parking function (in the classical sense) having a specific \textit{outcome permutation}, where the outcome permutation is the final order of the parked cars at the end of the classical parking process.
In particular, we will consider parking functions whose outcome is the identity permutation, a set which we denote $\idPF$.
For example, $(1,2,1)\in \PF_3(\id)$, but $(2,1,3)\notin\PF_3(\id)$ since the cars do not park in the order of the identity permutation.

We will also discuss statistics of weakly increasing parking functions, which are parking functions whose entries are weakly increasing.
More formally, we have the following definition.
\begin{definition}[Weakly Increasing Parking Function]\label{def:WeaklyIncreasingParkingFunction}
    A weakly increasing parking function of length $n \geq 1$ is a preference list $\alpha = (\alpha_1, \alpha_2, \ldots, \alpha_n) \in [n]^n$ such that  $\alpha_i \leq i$ for all $i \in [n]$ and $\alpha_1 \leq \alpha_2 \leq \cdots \leq \alpha_n$. 
    We denote the set of weakly increasing parking functions of length $n \geq 1$ by $\NDPF_n$.
\end{definition}
Weakly increasing parking functions are known Catalan objects (i.e., combinatorial objects enumerated by the Catalan numbers, for more on Catalan objects see \cite{catalan_stanley}). 
Moreover, the outcome of any weakly increasing parking function is the identity permutation (i.e., $\NDPF_n\subseteq \idPF$).
However, note that the converse is not true: for example, $(1,2,1) \in \PF_3$ parks in the order $123$, but is not weakly increasing.

The classical parking process, and resultant parking functions, have steadily increased in popularity as objects of study in the probability community.
Recently there have been several works studying phase transitions in the parking process on random graphs such as random trees \cite{aldous2023parking, contat2022sharpness,contat2023parking,contat2023parking1,contat2025parking,curien2019phase, goldschmidt2019parking}, and Frozen Er\H{o}s-R\'{e}nyi graphs \cite{ContatCurien}.
Diaconis and Hicks \cite{diaconis2017probabilizing}, on the other hand, considered the question \textit{``What does a random parking function look like?''}.
They gave probabilistic interpretations of some known generating functions, and used these results to derive statistics about parking functions. 

However, there are relatively few works which introduce randomness into the parking protocol itself. 
Durmi\'{c}, Han, Harris, Ribeiro, and  Yin \cite{durmic2022probabilistic} give one such model. 
In their parking protocol, whenever a car attempting to park finds its preferred parking spot occupied, with probability $p$ the car proceeds as in the classical parking protocol, and parks in the first empty spot it finds to the right. With probability $q\coloneqq 1-p$ the car reverses the classical parking protocol, and searches backwards in an attempt to park in the first open spot it finds to the left. 
Remarkably, the probability of obtaining a parking function is shown to be independent of the forward probability $p$. 
They also provided some combinatorial consequences of the probabilistic parking protocol, which gave interpretations of the triangular arrays produced by a generating polynomial related to the study of duplicial algebras and Lagrange inversion \cite[Theorem 8]{durmic2022probabilistic}. 

Motivated by the results of Diaconis and Hicks~\cite{diaconis2017probabilizing} and Durmi{\'c} et al.~\cite{durmic2022probabilistic}, we study statistics of a model of stochastic parking when it is started with an initial preference list which is a parking function whose outcome is the identity permutation. 
Note, that there are $n!$ parking functions with outcome the identity permutation, namely those preferences $\alpha \coloneqq (\alpha_1,\alpha_2,\ldots,\alpha_n)$ satisfying $a_i\leq i$ for all $i\in[n]$. For a proof that this property indeed characterizes the elements of $\idPF$, see Lemma~\ref{lem:PFWhoseOutcomeIsIdentity} below.  
Our contributions can be broadly categorized into two themes:
\begin{enumerate}
    \item
    Analyze properties of two different models of stochastic parking which we describe in detail in Section~\ref{sec:IntroducingTheModels} below, and \label{it: 1}
    \item 
    study statistics of the entries of uniformly random weakly increasing parking functions. \label{it: 2}
\end{enumerate}
We show that, when considering the set of parking functions whose outcome permutation is the identity permutation, both of our models of stochastic parking are almost completely characterized by the classical \textit{gambler's ruin} problem (for an overview of the subject, see \cite{feller1968introduction,grinstead2006grinstead,levinPeresWilmer}, which offer excellent expositions).
Therefore, our goal with \eqref{it: 2} is to characterize properties of a subset of these parking functions which allow us to obtain succinct expressions for \eqref{it: 1}.

In the remainder of this section we begin by giving a technical description of the models we consider. 
This is followed by the statements of our main contributions.
We conclude by relating our randomized model of parking with other more classical stochastic processes in the literature.

\subsection{Models of Stochastic Parking}\label{sec:IntroducingTheModels}
We now introduce our models of stochastic parking.
Let $n \in \mathbb{N}$ be the number of cars and consider a preference list $\alpha = (\alpha_1, \alpha_2, \ldots, \alpha_n) \in [n]^n$.
As in the classical model, cars arrive one at a time, and a car may enter the street only after its predecessor has completed its search process.
Upon the arrival of car $i \in [n]$, it attempts to park in its preferred spot $\alpha_i \in [n]$.
If spot $\alpha_i$ is unoccupied, the car parks there.
Otherwise, with probability $p\in [0,1]$, it displaces one spot to the right and attempts to park in spot $\alpha_i + 1$ so long as $\alpha_i + 1 \in [n]$.
Conversely, with probability $q\coloneqq1 - p$, it displaces one spot to the left and attempts to park in spot $
\alpha_i - 1$ so long as $\alpha_i -1 \in [n]$.
In this way, car $i$ performs a random walk starting from its preferred spot $\alpha_i$, and parking in the first unoccupied spot on the line segment $[n]$ it encounters, if any.
Note that the case in which $p = 1$ reduces to the original protocol of Konheim and Weiss~\cite{konheimWeiss}, while the case of $p=0$ corresponds to running the classical protocol from the right end of the segment to the left.
Now, we distinguish two different variants of this model depending on the dynamics at the \emph{boundaries}, namely spots $1$ and $n$:
 \begin{enumerate}[label=(\alph*)]
    \item 
    \emph{Unbounded model.}
    Cars are free to drift around the infinite set $\mathbb{Z}$.
    However, cars can only park in the finite set $[n]$.
    \item
    \emph{Open boundary.}
    Cars are free to drift around the finite set $[n]$.
    If a car exits this set (i.e., if it enters the set $\mathbb{Z} \setminus [n]$), it immediately terminates its search process and is unable to park.
    In particular, if a car is in spot $n$ (resp., spot $1$) and randomly displaces to spot $n + 1$ (resp., spot $0$), it cannot re-enter the set $[n]$.
\end{enumerate}
Our question of interest in this context is the following.
\begin{question}\label{quest:CentralQuestion}
Fix $\alpha \in [n]^n$ and $\sigma \in S_n$. What is the probability that the probabilistic parking protocol iniated with preferences $\alpha$ allows all of the cars to park with outcome $\sigma$?
\end{question}

\begin{figure}[ht]
    \centering
\begin{tikzpicture}[>=stealth, thick, scale = .7]

\draw (-3,0) ellipse (2.5 and 5);   
\draw ( 3,0) ellipse (2.5 and 5);   
\node at (-6,4) {$[n]^n$};
\node at ( 6,4) {$S_n$};

\draw (-3,2.5) ellipse (1 and 1.0);
\node at (-3,4) {\small $\{\alpha_i \le i, \forall i\}$};

\draw (-3,-3) ellipse (1 and 1.0);
\node at (-3,-1.6) {\small $\{\alpha_i \ge n+1-i, \forall i\}$};

\node[circle,fill=black,inner sep=1.5pt
] (a1) at (-3, 3) {};
\node[circle,fill=black,inner sep=1.5pt,
] (a2) at (-3, 2) {};
\node[circle,fill=black,inner sep=1.5pt
,label=left:$\alpha$
] (a5) at (-3,0) {};

\node[circle,fill=black,inner sep=1.5pt,
](a3) at (-3,-3.5) {};
\node[circle,fill=black,inner sep=1.5pt
] (a4) at (-3,-2.5) {};

\node[circle,fill=black,inner sep=1.5pt,label=right:$id$] (b1) at ( 3, 3.5) {};
\node[circle,fill=black,inner sep=1.5pt,label=right:$\sigma$] (b2) at ( 3, 1.5) {};
\node[circle,fill=black,inner sep=1.5pt,label=right:$\sigma'$] (b4) at ( 3, 0.4) {};
\node[circle,fill=black,inner sep=1.5pt,label=right:$\sigma''$] (b5) at ( 3, -.7) {};
\node[circle,fill=black,inner sep=1.5pt,label=below:$\mathrm{Rev}(id)$] (b3) at ( 3,-3) {};

\draw[->] (a1) -- (b1); 
\draw[->] (a2) -- (b1); 
\draw[->] (a3) -- (b3); 
\draw[->] (a4) -- (b3); 
\draw[->] (a5) -- (b2); 
\draw[->] (a5) -- (b4); 
\draw[->] (a5) -- (b5); 

\end{tikzpicture}
\caption{
   Illustration of the probabilistic parking protocol as a random mapping from $[n]^n$ to $S_n$. The dynamics induce a distribution over outcome permutations, and our main question of interest is to characterize this distribution $\Pr\left[ \operatorname{Out}(\alpha) = \sigma\right]$ where $\alpha$ and $\sigma$ are fixed, and we use $\operatorname{Out}(\alpha)$ to denote the outcome of $\alpha$ under the probabilistic parking protocol.
}
\label{fig:Mapping}
\end{figure}

\subsection{Contributions}\label{sec:contributions}
In our work we consider two subsets of parking functions, the set of weakly increasing parking functions and a superset consisting of those parking functions with outcome the identity permutation, which can be characterized as follows.

\begin{lemma}\label{lem:PFWhoseOutcomeIsIdentity}
    A preference vector $\alpha=(\alpha_1,\alpha_2,\ldots,\alpha_n)\in[n]^n$ satisfies $\alpha \in \PF_n(\id)$ if and only if its entries satisfy $\alpha_i \leq i$ for all $i \in [n]$.
\end{lemma}

\begin{proof}
  Assume $\alpha_i > i$ for some $i$. This contradicts the assumption that  $\alpha$ is a parking function whose outcome is the identity permutation.

    Conversely, an alternate characterization of parking functions says that, $\alpha \in [n]^n$ is a parking function if and only if the number of entries of $\alpha$ with value at most $i$ is at least $i$ for all $i\in [n]$. That is, 
    \[
    \left|\left\{k: \alpha_k \leq i\right\}\right| \geq i, \quad \quad \mbox{for all } 1 \leq i \leq n \eqpd 
    \]
    Thus, necessarily, if $\alpha_i \leq i$ for all $ i \in [n]$, $\alpha$ must be a parking function of length $n$. Furthermore, if $\alpha_i \leq i$ for every $i \in [n]$, then we may proceed by strong induction. It is clear that $\alpha_1 = 1$. Next, assume that for some $1 \leq j < n$, every $k$th car, where $1 \leq k \leq j$ parks in spot $k$. Then since $\alpha \in \PF_n$ the $(k+1)$th car parks in some spot $k <  b_{k + 1} \leq k + 1$, hence $b_{k + 1} = k + 1$ and the result holds.
\end{proof}

Notice that, not only does this characterize the set of parking functions with outcome permutation being the identity in the classical parking protocol of Konheim and Weiss, but a similar proof shows that the same condition characterizes the set of preferences in which, if all of cars are to park in the probabilistic parking protocol (with either boundary condition), then they must do so in the order of the identity permutation. In other words, there are no other (and no fewer) preferences which would allow all of the cars to park in the identity configuration in the probabilistic setting defined above. See Figure ~\ref{fig:Mapping} for an illustration of the idea. For both of our probabilistic models, we give expressions for the probability of parking and expected time to park (i.e., a random variable analogue of the displacement statistic), assuming we start with a preference vector $\alpha \in \idPF$ and conditioned on ultimately parking. Namely, we answer Question \ref{quest:CentralQuestion}
for two subsets of preferences, those which are parking functions with outcome being the identity permutation, and those with outcome being the reverse of the identity.

In the following, for each $i\in[n]$, we let $d_i \coloneqq i - \alpha_i$ denote the deterministic displacement of car $i$ for a parking function $\alpha=(\alpha_1,\alpha_2,\ldots,\alpha_n) \in \idPF$.
In addition, we let $X_{\alpha}^i$ be an indicator random variable which is $1$ if car $i$ parks under the stated protocol and preference list $\alpha$, and $0$ otherwise. 
Similarly, we let $X_{\alpha}$ be an indicator random variable which is $1$ if all of the cars park under the stated protocol and preference list $\alpha$, and $0$ otherwise. Let $Z_t^i$ be the random variable keeping track of the position of car $i$ at time $t$, and for each car, we define the stopping time 
\[\tau_{\alpha}^{i} \coloneqq \min\{ t \geq 0\ : Z_{t}^i = i \mbox{ or } Z_{t}^i = \infty \} \eqcom \]
when considering the unbounded model and 
\[\tau_{\alpha}^{i} \coloneqq \min\{ t \geq 0\ : Z_{t}^i = i \mbox{ or } Z_{t}^i = 0 \} \eqcom  \] in the case of the model with open boundaries. Furthermore, let $\tau_\alpha$ denote the time it takes the entire protocol to complete under the specified boundary condition. In particular, we have 
\[
\E\left[\,\tau_{\alpha} \mid X_{\alpha}=1\right]
=\sum_{i=1}^n \E\left[\,\tau_{\alpha}^i \mid Z^i_{\tau^i_{\alpha}}=i \right]\eqpd
\]
Since we always assume that we start the probabilistic parking protocol using $\alpha \in \idPF$, we have that the first car is always a \textit{lucky car}. 
In general, any car that parks in its preferred parking spot immediately is called a lucky car. Thus, in particular, the first car to park (the car with preference $\alpha_1$) can always park in its preferred parking spot. 
Therefore
\[
    \Pr[X_{\alpha}^1 = 1] = 1  \ \ \  \mbox{ and }  \ \ \ \E[\tau^1_{\alpha}] = 0\eqpd
\]

In Section~\ref{sec:UnboundedModel}, we prove the following main result.
\begin{theorem}[Unbounded Model]\label{thm:unboundedModelIntro}
Let $\alpha = (\alpha_1, \alpha_2, \ldots, \alpha_n) \in \idPF$. 
In the probabilistic parking protocol under the unbounded model we have the following.
\begin{enumerate}
    \item 
    If $1 < i \leq n$, then
\begin{equation*}
    \Pr\left[X_\alpha^i = 1 \mid\prod_{j=1}^{i-1} X_\alpha^{j} = 1\right] 
    = 
    \sum_{\ell = 0}^\infty \frac{d_i}{\ell + d_i}\binom{2\ell + d_i - 1}{\ell} p^{d_i + \ell}q^{\ell}\eqpd
\end{equation*}
Moreover, $\Pr\left[X_\alpha^i = 1 \mid\prod_{j=1}^{i-1} X_\alpha^{j} = 1\right] = 1$ if and only if $1/2 \leq p \leq 1$.

    \item 
    If $1 < i \leq n$ and $1/2 < p \leq 1$, then
\begin{equation*}
    \E\left[\,\tau_\alpha^i \mid \prod^{i}_{j=1}X_\alpha^{j} = 1\right] 
    = 
    \frac{d_i}{p-q}\eqpd
\end{equation*}
\end{enumerate}
\end{theorem}
In Section~\ref{sec:OpenBoundary}, we consider the open boundary model and prove the following main result.
\begin{theorem}[Open Boundary]\label{thm:openboundaryModelIntro}
Let $\alpha = (\alpha_{1}, \alpha_2, \ldots, \alpha_n) \in \idPF$.
In the probabilistic parking protocol with open boundary condition we have the following.
\begin{enumerate}
    \item If $1 < i \leq n$, then
\begin{equation*}
     \Pr\left[X_\alpha^i = 1 \mid \prod_{j=1}^{i-1}X_\alpha^{j} = 1\right] = 
    \begin{cases}
      \displaystyle\frac{
    \alpha_i}{i}
     &\mbox {if } \ \ p = \frac{1}{2}\\
    
        \displaystyle p^{i-\alpha_i}\frac{p^{\alpha_i}-q^{\alpha_i}}{p^i-q^i} \, & \mbox{otherwise}\eqpd
    \end{cases}
\end{equation*}
\item
If $1 < i \leq n$, then
\begin{equation*}
\E\left[\,\tau_{\alpha}~| ~ \prod^{i}_{j=1}X_\alpha^{j} = 1\right] =
 \begin{cases}
    \displaystyle\frac{1}{3}(i^2 - \alpha_i^2) 
     &\mbox {if } \ \ p =\frac{1}{2}\\
      \displaystyle\frac{1}{p-q}\frac{
    i\left(1 + \left(\frac{q}{p}\right)^i\right)
}{ \left(1 - \left(\frac{q}{p}\right)^i\right)} - \frac{\alpha_i \left(1 + \left(\frac{q}{p}\right)^{\alpha_i}\right)}{\left(1 - \left(\frac{q}{p}\right)^{\alpha_i}\right)}  \, & \mbox{otherwise}\eqpd
    \end{cases}
\end{equation*}
\end{enumerate}
\end{theorem}

The proofs of Theorems~\ref{thm:unboundedModelIntro} and~\ref{thm:openboundaryModelIntro} are based on gambler's ruin arguments.
 Theorem~\ref{thm:unboundedModelIntro} in particular invokes results from McMullen~\cite{mcmullen2020drunken}. Notice, in contrast to the model of Durmi\'{c} et al.,~\cite{durmic2022probabilistic} the probability that all of the cars park in our case retains the parameter~$p$.
Moreover, in Section~\ref{sec:NegativeCorrelation} we prove that the random variables $X_{\alpha}^i$ are negatively correlated.

Next, we state our main results for theme (\ref{it: 2}).
In Section~\ref{sec:samplingWipfs}, we analyze the distribution of the last entry of a uniformly random weakly increasing parking function, as well as its expected value. In particular, we show the following.

\begin{lemma}\label{lem:distributionOfAlphaNBackwards} Let $\alpha \sim \U(\PF_n)$ be sampled uniformly at random from the set of all weakly increasing parking functions of length $n \in \N$. For $n$ sufficiently large, and fixed $j \geq 1$,
\be\label{eq:distributionOfAlphaNBackwards}
\Pr_{\alpha \,\sim\, \mathcal{U}(\PF_n)}\left[\alpha_n= n-j \right] \sim \frac{1}{\sqrt{2}}\frac{j-1}{
4^n}.
\ee
\end{lemma}

\begin{lemma}
If $\alpha \sim \mathcal{U}\left(\NDPF_n\right)$ is a uniformly random weakly increasing parking function of length $n\in\N$, then the expected value for the last entry is given by
$$
\E\left[\,\alpha_n\right] = \frac{n(n-1)}{n+2}.
$$
\end{lemma}

In this section, we also give the probability that a uniformly random weakly increasing parking function of length $n\in \N$ has a specific fixed set, $L \subseteq [n]$, of lucky cars, as well as the probability that a uniformly random weakly increasing parking function has a specific fixed number of lucky cars. 
Finally, we conclude with Section~\ref{sec:future work}, where we give some directions for future research.

\subsection{Related Work}\label{relatedWork}
The probabilistic parking protocol described can easily be interpreted as an \textit{interacting particle system}
(for a thorough treatment of the subject, refer to \cite{Liggett_2004,SPITZER1970246}). 
In particular, preference vectors translate directly to initial configurations of particles (i.e., cars) on the line segment such that multiple particles are allowed on the same site.

The underlying Markov chain would have state space $[n]^n$, and the set of initial configurations are those allowing each site to contain a non-negative number of labeled particles, with the set of absorbing states being those in which each site contains zero or one particle.

More specifically, one can interpret the dynamics of our probabilistic parking process as being a generalization of the model known as Internal Diffusion Limited Aggregation, or IDLA, which was originally introduced by Diaconis and Fulton \cite{diaconis1991growth,lawler}. (IDLA is also known to be essentially equivalent to a special case of the Activated Random Walk model when the rate of a particle sleeping $\lambda$ is set to be infinite \cite{rolla2020activated}.)These models are of interest to the interacting particle system community as they exemplify a phenomena called \textit{self-organized criticality} in which the system appears to undergo a phase transition, not by way of fine tuning the parameters of the model, but rather as a natural result of the dynamics of the process.

Our model differs from this more classical model, however, as we enforce that the cars (or particles) remain distinct throughout the process, and we are interested in the positions of the labeled cars at the end of the process. That is we give a total order to the labels of the cars attempting to park, and so when several cars are assigned to the same initial spot, the car with the smallest number in the order is allowed to park immediately while other cars must perform random walks in increasing order of their labels. One can then see our result then as working towards characterizing the \textit{stationary distribution} of the Markov chain. In general, computing stationary measures for various interacting particle systems is a classical, and still lively, area of research at the intersection of probability theory and combinatorics (e.g., \cite{bramson2002characterization, CorteelSylvie2011Tcft,MartinStationary, Nestoridi_Schmid_2024,yang2023stationarymeasureshigherspin}).

Recently and independently Varin introduced the Golf Model \cite{VarinGolf}, an interacting particle system that can be seen as a generalization of IDLA. In their work, Varin also introduces a probabilistic parking protocol on the cycle $\Z/n\Z$ as a special case of the Golf Model in which each of the, now unlabeled, cars starts at a uniformly random spot on the circular parking lot.
The cars move to the right with probability $p$ and to the left with probability $1-p$ to find the nearest open parking spot, as in our process.
The focus of our works differ, however, as Varin is interested in the setting in which there are strictly more parking spots than cars. Thus, they are interested in studying the distribution of the set of unoccupied sites at the end of the process. 

In publishing this work, we were made aware that Nadeau and Tewari also investigate a general form of IDLA on $\Z$, which is essentially equivalent to our model with open boundaries. In \cite{nadeau2021qdeformationalgebraklyachkomacdonalds} the authors study at a generalization of the classical IDLA model with general jump rates, and determine the probability that the final configuration of the unlabeled particles on the line is to occupy the set of sites in $[n]$. 
The authors approach is strictly algebraic, and they are instead interested in enumerative results related to this probability, obtaining an expression for it in terms of \textit{remixed Eulerian numbers} \cite{NadeauRemixedEulerian}. More recently (and independently) Nadeau formalized this interacting particle as a parking procedure in \cite{nadeau:hal-04262134}.\\

Overall, one can view our contribution to this literature as investigating the model from a probabilistic perspective and obtaining explicit expressions for the probability of parking in terms of the parameters $p$ and $\alpha$, while still maintaining an understanding of where the cars end up parking. This allows us to adopt the perspective of the parking process as a random mapping from $[n]^n$ to $S_n$. Moreover we offer new computations for the expected time to park, as well as introduce new boundary conditions for the process.


\section{Sampling Weakly Increasing Parking Functions and Lucky Cars}\label{sec:samplingWipfs}

It is well known that the number of weakly increasing parking function of length $n$ is given by the $n$th Catalan number $C_n = \frac{1}{n+1}\binom{2n}{n}$.
Therefore, weakly increasing parking functions are in bijection with the wide variety of  \emph{Catalan objects}~\cite{stanley2015catalan}.
For example, weakly increasing parking functions of length $n$ are in bijection with Dyck paths of length $2n$, as well as with binary trees with $n + 1$ leaves.
We recall that a Dyck path is a diagonal lattice path from the origin $(0,0)$ to the $x$-axis, consisting of only ``up steps'' $U=(1,1)$ and ``down steps'' $D=(1,-1)$. 
A Dyck path of length $2k \in \mathbb{N}$ involves $k$ up steps and $k$ down steps, and ends at $(0,2k)$.
A classical bijection between Dyck paths of length $2n$ and weakly increasing parking functions of length $n$ sets the $i$th entry of the preference list to be $1$ plus the number of ``down steps'' occurring before the $i$th ``up step.''
We illustrate this bijective mapping in Figure~\ref{fig: dyck path}.
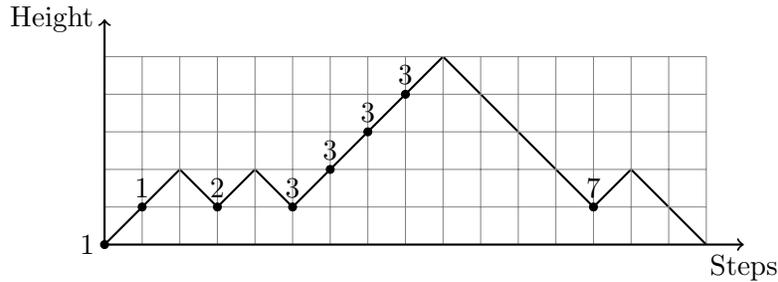
\begin{figure}[ht]
    \centering
\begin{tikzpicture}[scale=0.5]
    \draw[thick] (0,0) -- (1,1); 
    \draw[thick] (1,1) -- (2,2); 
    \draw[thick] (2,2) -- (3,1); 
    \draw[thick] (3,1) -- (4,2); 
    \draw[thick] (4,2) -- (5,1); 
    \draw[thick] (5,1) -- (6,2); 
    \draw[thick] (6,2) -- (7,3); 
    \draw[thick] (7,3) -- (8,4); 
    \draw[thick] (8,4) -- (9,5); 
    \draw[thick] (9,5) -- (10,4); 
    \draw[thick] (10,4) -- (11,3); 
    \draw[thick] (11,3) -- (12,2); 
    \draw[thick] (12,2) -- (13,1); 
    \draw[thick] (13,1) -- (14,2); 
    \draw[thick] (14,2) -- (15,1); 
    \draw[thick] (15,1) -- (16,0); 

    \draw[help lines] (0,0) grid (16,5);
    \draw[thick,->] (0,0) -- (17,0) node[anchor=north] {Steps};
    \draw[thick,->] (0,0) -- (0,6) node[anchor=east] {Height};
    \filldraw (1,1) circle (3pt); 
    \node[anchor=south] at (1,1) {1};
    \filldraw (3,1) circle (3pt); 
    \node[anchor=south] at (3,1) {2};
    \filldraw (5,1) circle (3pt); 
    \node[anchor=south] at (5,1) {3};
    \filldraw (6,2) circle (3pt); 
    \node[anchor=south] at (6,2) {3};
    \filldraw (7,3) circle (3pt); 
    \node[anchor=south] at (7,3) {3};
    \filldraw (8,4) circle (3pt); 
    \node[anchor=south] at (8,4) {3};
    \filldraw (13,1) circle (3pt); 
    \node[anchor=south] at (13,1) {7};
    \filldraw (0,0) circle (3pt); 
    \node[anchor=east] at (0,0) {1};    
\end{tikzpicture}
\caption{
    The Dyck path UUDUDUUUUDDDDUDD, of length $16$ corresponds to the weakly increasing parking function $(1,1,2,3,3,3,3,7) \in \NDPF_8$.
}
\label{fig: dyck path}
\end{figure}

Therefore, to sample a uniformly random weakly increasing parking function, one can use of one of the many methods to generate uniformly random Catalan objects, such as R\`emy's algorithm~\cite{Remy}.
Refer to  \cite{diaconis2017probabilizing} for an algorithm to sample uniformly from the set of all parking functions.

\subsection{Distribution of the Last Entry of a Weakly Increasing Parking Function}

In this section, we give asymptotic results for the distribution of the last entry of a uniformly random weakly increasing parking function. 
First, we consider the probability that a uniformly random parking function is weakly increasing. 
\begin{lemma}
Let $\alpha \sim \U(\PF_n)$ be sampled uniformly at random from the set of parking functions of length $n \in \N$.
If $n$ is sufficiently large, then
\[
    \Pr_{\alpha\, \sim \,\U(\PF_n)}\left[\alpha \in \NDPF_n\right] \sim \frac{1}{e\sqrt{\pi n}} \left( \frac{4}{n}\right)^n.
\]
\end{lemma}

\begin{proof}
    Recall that $|\PF_n| = (n+1)^{n-1}$ and $|\NDPF_n| = \frac{1}{n+1} \binom{2n}{n}$.
    Therefore,
    \begin{align*}
        \Pr_{\alpha\, \sim \,\U(\PF_n)}\left[\alpha \in \NDPF_n\right]
        &= \frac{\frac{1}{n+1} \binom{2n}{n}}{(n+1)^{n-1}}
        = \frac{1}{(n+1)^n} \frac{(2n)!}{(n!)^2}
        \stackrel{(*)}{\sim} \frac{1}{(n+1)^n} \frac{\sqrt{2\pi (2n)} \left(\frac{2n}{e} \right)^{2n}}{(2\pi n)\left( \frac{n}{e} \right)^{2n}} \label{eq: stirlings} \\
        &= \frac{1}{(n+1)^n} \frac{4^n}{\sqrt{\pi n}}
        = \frac{1}{n^n \left(1 + \frac{1}{n}\right)^n} \frac{4^n}{\sqrt{\pi n}}
        = \frac{1}{e\sqrt{\pi n}} \left( \frac{4}{n}\right)^n,
    \end{align*}
    where $(*)$ follows from Stirling's approximation of the factorial.
\end{proof}

Diaconis and Hicks, in \cite{diaconis2017probabilizing}, study the distribution of the first entry of a uniformly random parking function. 
As motivated by our results stated in Section~\ref{sec:contributions}, an analogous study specifically for weakly increasing parking function is also of interest.
However, in the weakly increasing setting, the first entry is always equal to $1$.
Therefore, in what follows we focus on studying the distribution of the last entry $\alpha_n$ of a uniformly random weakly increasing parking function.

In order to compute distributional information, we establish the following result which has been simultaneously and independently derived by Butler et al. \cite{butler2024lucky}.

\begin{proposition}\label{prop:NumberOfWipf} 
The number of weakly increasing parking functions $\alpha \in \NDPF_n$ with $\alpha_i = j$ for fixed $1\leq j\leq i\leq n$ is given by
    \be\label{eq:NumberOfWIPF}
    \frac{(i-j+1)(i-j+2)}{i(n-j+2)} \binom{i+j-2}{j-1} \binom{2n-i-j+1}{n-i} \eqpd
    \ee
\end{proposition}
\begin{proof}
Let $\alpha = (\alpha_1, \alpha_2, \ldots, \alpha_n) \in \NDPF_n$ with $\alpha_i = j$ for fixed $j \leq i \in [n]$.
Now, consider all the different preference lists $(\alpha_1, \alpha_2, \ldots, \alpha_i)$ and $(\alpha_{i+1}, \alpha_{i+2}, \ldots, \alpha_n)$ that can be concatenated to form $\alpha$:
\begin{itemize}
    \item 
    If $\alpha \in \NDPF_n$ and $\alpha_i = j$, then $(\alpha_1, \alpha_2, \ldots, \alpha_i) \in \NDPF_i$ with maximum entry $j$.
    It is well-known that the number of weakly increasing parking functions of length $n + 1$ with maximum entry $k + 1$ is given by $C(n,k) = \frac{n-k+1}{n+1} \binom{n+k}{k}$, i.e. the $(n, k)$th entry of the Catalan triangle (OEIS \seqnum{A009766}).
    Therefore, there are
    \begin{equation*}
        C(i-1,j-1) = \frac{i - j + 1}{i}\binom{i + j - 2}{j-1}
    \end{equation*}
    possible choices for the preference list $(\alpha_1, \alpha_2, \ldots, \alpha_i)$.
    \item
    If $\alpha \in \NDPF_n$, then $\alpha$ is in bijection with a Dyck path of length $2n$.
     Furthermore, if $\alpha_i = j$ then, in particular, the preference lists $(\alpha_{i+1}, \alpha_{i+2}, \ldots, \alpha_n)$ are in bijection with North-East lattice paths from $(i+1, n-j+1)$ to $(n, n)$ that do not cross above the main diagonal. After a sequence of transformations one can obtain that these paths are in bijection with Dyck paths that start from $(0,0)$ and end at $(n-k+1,n-i)$.
    It is again well-known that these are counted by entries of the Catalan triangle \cite{DyckPathEnumeration1999}, specifically by
    \begin{equation*}
        C(n-j+1,n-i) = \frac{i - j + 2}{n-j+2}\binom{2n-i-j+1}{n-i} \eqpd
    \end{equation*}
    This gives the number of possible choices for the preference list $(\alpha_{i+1}, \alpha_{i+2}, \ldots, \alpha_n)$.
\end{itemize}
Since the preference lists are independent, the total number of choices is the product of these two counts. This completes the proof.
\end{proof}

Given Proposition~\ref{prop:NumberOfWipf}, we can now give the following probability of a uniformly random weakly increasing parking function having a particular fixed value as its last entry. 
\begin{lemma}\label{lem:distributionOfAlphaNFixed}
Let $n \in \N$, $j\geq 1$, and $X$ be a Poisson random variable with mean $\lambda = n+j-2$.
That is,
$$\Pr\left[X=k\right]=\frac{\lambda^k e^{-\lambda}}{k!} \eqpd$$ 
Let $\alpha \sim \U(\NDPF_n)$ be sampled uniformly at random from the set of all weakly increasing parking functions of length $n \in \N$. 
If $n$ is sufficiently large and $j \geq 1$ is fixed, then

\[\label{eq:distributionOfAlphaNFixed}
 \Pr_{\alpha\, \sim \,\U(\NDPF_n)}\left[\alpha_n=j \right] \sim \frac{\sqrt{\pi n} (n-j+1)e^{n+j-2}}{4^n}\Pr[X=j-1] \eqpd
\]

\end{lemma}
\begin{proof}
    Using Equation~\eqref{eq:NumberOfWIPF} in Proposition~\ref{prop:NumberOfWipf}, we have 
    \begin{align*}
        \Pr_{\alpha \,\sim\, \U(\NDPF_n)}\left[\alpha_n=j \right] 
        &=
        \frac{\frac{(n-j+1)(n-j+2)}{n(n-j+2)} \binom{n+j-2}{j-1} \binom{2n-n-j+1}{n-n}}{\frac{1}{n+1} \binom{2n}{n}}\\ 
        &\sim (n-j+1)\binom{n+j-2}{j-1}\binom{2n}{n}^{-1}\\
        &\stackrel{(*)}{\sim} \frac{\sqrt{\pi n}(n-j+1)}{4^n}\binom{n+j-2}{j-1} \\
        &\stackrel{(\dagger)}{\sim} \frac{\sqrt{\pi n}(n-j+1)}{4^n}\frac{(n+j-2)^{j-1}}{(j-1)!} \\
        &= \frac{\sqrt{\pi n}(n-j+1)e^{n+j-2}}{4^n}\Pr[X=j-1],
    \end{align*}
    where $(*)$ and $(\dagger)$ follow from Stirling's approximation of the factorial.
\end{proof}

Additionally,  we can prove Lemma \ref{lem:distributionOfAlphaNBackwards}, which we restate below.

\begin{lemma} Let $\alpha \sim \U(\NDPF_n)$ be sampled uniformly at random from the set of all weakly increasing parking functions of length $n \in \N$. For $n$ sufficiently large, and $j\geq 1$ fixed with respect to $n$,
\[
\Pr_{\alpha \sim \U(\PF_n)}\left[\alpha_n= n-j \right]\sim \frac{1}{\sqrt{2}}\frac{j-1}{
4^n}\eqpd
\]
\end{lemma}
\begin{proof}A similar argument as that in Lemma~\ref{lem:distributionOfAlphaNFixed} and using Proposition~\ref{prop:NumberOfWipf} gives

\begin{align*}
\begin{split}
    \frac{(n+1)(j-1)}{n}\binom{2n-j-2}{n-j-1}\binom{2n}{n}^{-1}
    &=(j-1)\frac{(2n-j-2)!n!n!}{(n-j-1)!(n-1)(2n)!}\\
&\hspace{-1.25in}\stackrel{(\ast)}{\sim}{(j-1)}\frac{(2n-j-2)^{2n-j-2+\frac{1}{2}}\sqrt{2\pi}n^{n+\frac{1}{2}} \sqrt{2\pi}n^{n+\frac{1}{2}}\sqrt{2\pi}e^{4n+j+2}}{(n-j-1)^{n-j-1+\frac{1}{2}}\sqrt{2\pi}(n-1)^{n-1} \stp(2n)^{2n+\frac{1}{2}}\stp e^{4n+j+2}}\\
&\hspace{-1.25in}\sim{(j-1)}\frac{(2n)^{2n-j-2+\frac{1}{2}}(1-\frac{j}{2n}-\frac{1}{n})^{2n-j-1+\frac{1}{2}}n^{n+\frac{1}{2}} n^{n+\frac{1}{2}}}{n^{n-j-1+\frac{1}{2}}(1-\frac{j}{n}-\frac{1}{n})^{n-j-1+\frac{1}{2}}n^{n-1+\frac{1}{2}}(1-\frac{1}{n})^{n-1+\frac{1}{2}} (2n)^{2n+\frac{1}{2}} }\\
&\hspace{-1.25in}\sim{(j-1)}\frac{e^{-j-2}}{e^{-j-2}2^{2n+\frac{1}{2}}}\\
&\hspace{-1.25in}\sim\frac{j-1}{4^n\sqrt{2}}\\
\end{split}
\end{align*}
where $(*)$ follows from Stirling's approximation of the factorial.
\end{proof}

The following result holds in the case that $j$ goes to infinity with~$n$.

\begin{lemma}\label{lem:distributionOfAlphaNGrowing}
Let $X$ be a Poisson distributed random variable with mean $\lambda = n+j-2$.  That~is 
$$\Pr[X=k]=\frac{\lambda^k e^{-\lambda}}{k!} .$$ 
Let $\alpha \sim \U(\NDPF_n)$ be sampled uniformly at random from the set of  all weakly increasing parking functions of length $n \in \N$. For $n$ and $j$ both going to infinity, we have 
\[
\Pr_{\alpha \,\sim\, \U(\PF_n)}\left[\alpha_n=j \right] \sim \frac{(n+1)e^{n+j-1}(n-j+1)\sqrt{2\pi n}}{(n+j-2)4^{n}}\Pr[X=j-1]\eqpd
\] 
\end{lemma}
\begin{proof}
    Using Equation~\eqref{eq:NumberOfWIPF} in Proposition~\ref{prop:NumberOfWipf}, we have 
    \begin{align*}
    \frac{(n+1)(n-j+1)}{n}\binom{n+j-2}{j-1}\binom{2n}{n}^{-1}
    &= {(n+1)(n-j+1)}\frac{(n+j-2)!n!}{(j-1)!(2n)!}\\
&\hspace{-1.25in}\stackrel{(\ast)}{\sim}{(n-j+1)(n+1)}\frac{(n+j-2)^{n+j-2+\frac{1}{2}}\sqrt{2\pi}n^{n}\sqrt{2\pi n}e^{2n}}{\sqrt{2\pi}(j-1)!(2n)^{2n+\frac{1}{2}}e^{2n+j-2}}\\
 &\hspace{-1.25in}\sim\frac{(n+1)(n-j+1)n^{n+\frac{1}{2}}(1+\frac{j-2}{n})^{n+\frac{1}{2}}(n+j-2)^{j-2}n^{n}\sqrt{2\pi n}e^{2n}}{(j-1)!4^{n+\frac{1}{2}}(n)^{2n+\frac{1}{2}}e^{2n+j-2}}\\
&\hspace{-1.25in}\sim\frac{(n+1)(n-j+1)e^{j-2}(n+j-2)^{j-2}\sqrt{2\pi n}}{(j-1)!4^{n+\frac{1}{2}}e^{j-2}}\\
&\hspace{-1.25in}=\frac{(n+1)e^{n+j-1}(n-j+1)\sqrt{2\pi n}}{(n+j-2)4^{n+\frac{1}{2}}}\Pr[X=j-1]
    \end{align*}
where, once more, $(*)$ follows from Stirling's approximation of the factorial.
\end{proof}
Next we give the expected value of the last entry of a uniformly random weakly increasing parking function.

\begin{lemma}\label{lem:expectationOfAlphaSubNInWIPF}
If $\alpha \sim \mathcal{U}\left(\NDPF_n\right)$ is a uniformly random weakly increasing parking function of length $n$, then the expected value for the last entry is given by
$$
\E\left[\,\alpha_n\right] = \frac{n(n-1)}{n+2}.
$$
\end{lemma}
\begin{proof}
Again using Proposition \ref{prop:NumberOfWipf}, we get that the expected value of the last entry of a weakly increasing parking function of length $n$ is given by 
\[   
    \begin{split}
    \E\left[\,\alpha_n\right]&=
    {(n+1)}\binom{2n}{n}^{-1}\left[\sum\limits_{j=0}^{n-1}\frac{j(n-j)}{n}\binom{n+j-1}{j}\right]\\
    &={(n+1)}\binom{2n}{n}^{-1}\left[\sum\limits_{j=0}^{n-1}j\binom{n+j-1}{j}-\frac{1}{n}\sum\limits_{j=0}^{n-1}j^2\binom{n+j-1}{j}\right].
    \end{split}
    \]

Using the identities
\begin{align*}
\displaystyle\sum_{j=0}^{n-1} j\binom{j+n-1}{j}&=\frac{(n-1) n\binom{2 n-1}{n-1}}{n+1} 
\intertext{and}
\sum_{j=0}^{n - 1} j^2 \binom{j+n-1}{j} &= \frac{(n-1)n^3}{(n+1)(n+2)} \binom{2n-1}{n-1}, 
\end{align*}
whose proofs we defer to the appendix (in particular, refer to Proposition~\ref{eq:BinomialIdentity1}), we have 
  \begin{align*}
   \E\left[\,\alpha_n\right]&=   {(n+1)}\binom{2n}{n}^{-1}\left[\frac{n(n-1)}{(n+1)}\binom{2n-1}{n-1}-\frac{1}{n}\sum\limits_{j=0}^{n-1}j\binom{n+j-1}{j}\right]\\
    &=   {(n+1)}\binom{2n}{n}^{-1}\left[\frac{n(n-1)}{(n+1)}\binom{2n-1}{n-1} - \frac{(n-1)n^2}{(n+1)(n+2)} \binom{2n-1}{n-1}\right]\\
    &= (n(n-1))\binom{2n-1}{n-1}\binom{2n}{n}^{-1}\left[1 -  \frac{n}{(n+2)}\right]\\
    &=\frac{n(n-1)}{n+2}.
\qedhere    \end{align*} 
\end{proof}
In addition, due to the constraint that $\alpha_i\leq i$ for all $i \in n$, we have the following remark.
\begin{proposition}
Fix $m\in [n]$ and indices $(i_1,i_2,\ldots,i_k)\in[n]^k$ satisfying $i_1 \leq \dots\leq  i_k$ and $i_1\leq m \leq n$.
Then, for $\alpha\in \NDPF_n$, we have 
    \[\Pr\left[\alpha\,|\, \alpha_{i_1} \leq m, \alpha_{i_2} \leq m, \ldots, \alpha_{i_k} \leq m\right] =  \Pr[\alpha\,|\, \alpha_{i_k} \leq m].\] 
\end{proposition}
\begin{proof}
    The result follows since $\alpha$ is a weakly increasing parking function, hence $\alpha_{i_j} \leq \alpha_{i_k}$ for $i_j \leq i_k$.
\end{proof}

Furthermore, we remark that the probabilities given by Lemmas \ref{lem:distributionOfAlphaNFixed} and \ref{lem:distributionOfAlphaNGrowing} are monotonically decreasing in $j$. This can be seen through Lemma~\ref{lem:monotonicityofWIPF} whose proof can be found in the appendix.

\subsection{Lucky Cars in Weakly Increasing Parking Functions}
For completeness, in this subsection we bring together some recent enumerative results on  parking function stating them in the language of probability. Harris and Martinez \cite{harris2024parkingfunctionsfixedset} previously considered the following enumerative question: 
Given a 
subset $L\subseteq [n]$ containing $1$, how many parking functions in $\PF_n$ have their set of lucky cars being those indexed by the set $L$? Their results involve enumerating these parking functions based on the outcome permutation, namely the permutation which details which car parked in what spot. 
They then restrict to the set of weakly increasing parking functions and ask the analogous question. Their result in this case involves a product of Catalan numbers, see \cite[Theorem 2.8]{harris2024parkingfunctionsfixedset}. 

In what follows, we independently recover their result and use this to give expected values for the number of weakly increasing parking functions with a fixed set of lucky cars, and for a fixed number of lucky cars. 
To begin, we use the bijection between parking functions and Dyck paths to find the expected number of lucky cars in a uniformly random weakly increasing parking function.
\begin{remark}\label{remark1}
     We note that the number of returns in a Dyck path is given by the number of returns, which is the number of times the Dyck path touches the $x$-axis, excluding the endpoint $(0,2n)$, provided our Dyck path is non-empty. 
     Thus, the number of lucky cars in $\alpha\in\NDPF_n$ is equal to the number of returns in the Dyck path corresponding to $\alpha$.
\end{remark}
From Remark~\ref{remark1} we immediately get the following result.
\begin{corollary}
    The expected number of lucky cars in a weakly increasing parking function with $n$ cars is $\frac{3n}{n+2}$.
\end{corollary}
\begin{proof}
Deutsch~\cite{DyckPathEnumeration1999} showed that expected number of returns for a Dyck path of semilength $n$ (length $2n$) is $\frac{3n}{n+2}$. Thus, the expected number of lucky cars in a weakly increasing parking function is precisely $\frac{3n}{n+2}$.
\end{proof}
Next, we consider the following question: Given a set of lucky cars $L$, what is the probability a uniformly selected weakly increasing parking function $\alpha\in\mathcal{U}(\NDPF_n)$ has $L$ as its set of lucky cars.
We note that a car which is lucky with respect to the classical parking protocol of Konheim and Weiss is also lucky with respect to our probabilistic parking protocol, and thus, the following result is independent of the chosen parking protocol.

\begin{lemma}\label{lem:weaklyincreasingwithfixedsetofluckycars}
    Let $L_\alpha \subseteq [n]$ denote the set of lucky cars at the end of a parking process run with preference list $\alpha \in \NDPF_n$. Furthermore, let $L \coloneqq \{l_1, \dots, l_k\} \subseteq [n]$  such that  $l_1 \leq l_2 \dots \leq l_k$. Thus, we have that 
    \[\Pr_{\alpha \sim \mathcal{U}(\NDPF_n)}[L_\alpha= L] =\begin{cases} 0 & \text{if } 1\not\in L \\
    \displaystyle\frac{\prod_{i=1}^{k} C_{x_j}}{C_n}
    & \text{if } 1\in L 
    \end{cases}\]
    where $C_{x_j}$ is the $x_j$th Catalan number and $x_j \coloneqq l_{j+1}- l_j - 1$ for $1\leq j <k $, and $x_k \coloneqq n - l_k$.
\end{lemma}
\begin{proof}
The case of $1\not\in L$ is clear, as mentioned previously, assuming we run the parking protocol with a weakly increasing parking function, the first car is always lucky. When $1\in L$, the result follows from \cite[Theorem 2.8]{harris2024parkingfunctionsfixedset} which states that the number of weakly increasing parking functions with lucky cars in the set $L$ is exactly \[\displaystyle \prod_{j=1}^{k} C_{x_j}\]
    where $C_{x_j}$ is the $x_j$th Catalan number and $x_j\coloneqq l_{j+1}- l_j - 1$ and $x_k \coloneqq n - l_k$ when $1\in L$. This result also follows from the bijection to Dyck paths; the number of parking functions with lucky cars given by the set $L$ are in bijection with the set of Dyck paths with returns with the set of coordinates $(2l-2,0)$ for each $l\in L\setminus \{1\}$. We can count the number of such Dyck paths by thinking of each Dyck paths as a composition of smaller Dyck paths with exactly 1 return and length equal to the difference in position of the returns in our original Dyck path (see Figure~\ref{fig:DyckPathDecomp}). Our result follows from the fact that the number of Dyck paths with semilength $l_{j+1} -l_{j}$ and exactly 1 return is precisely $C_{l_{j+1} - l_{j}}$.
    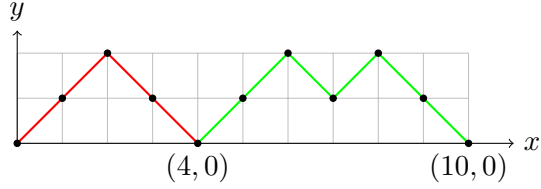
\begin{figure}

\begin{tikzpicture}[scale=0.6]
    \draw[lightgray, very thin] (0,0) grid (10,2);

    \draw[->] (0,0) -- (11,0) node[right] {$x$};
    \draw[->] (0,0) -- (0,2.5) node[above] {$y$};

    \draw[red, thick] 
        (0,0) -- (1,1) -- (2,2) -- (3,1) -- (4,0);
    
    \draw[green, thick]
        (4,0) -- (5,1) -- (6,2) -- (7,1) -- (8,2) -- (9,1) -- (10,0);

    \foreach \x/\y in {0/0, 1/1, 2/2, 3/1, 4/0, 5/1, 6/2, 7/1, 8/2, 9/1, 10/0} {
        \filldraw (\x,\y) circle (2pt);
    }

    \node[below] at (4,0) {$(4,0)$};
    \node[below] at (10,0) {$(10,0)$};
\end{tikzpicture}
\caption{The Dyck path corresponding to $(1,1,3,3,4)\in\NDPF_5$ can be thought of as a composition of the two different Dyck paths, each with no returns except the last.}
\label{fig:DyckPathDecomp}
\end{figure}
\end{proof}

Using Lemma \ref{lem:weaklyincreasingwithfixedsetofluckycars}, we can prove the following.
\begin{lemma}Let $|L_\alpha|$ denote the number of lucky cars at the end of a parking process run with preference list $\alpha \in \NDPF_n$, and suppose $1\leq k\leq n$. Then, we have 
    \[\Pr_{\alpha \,\sim\, \mathcal{U}(\NDPF_n)}\left[|L_\alpha|= k\right]=\frac{\binom{2n-k-1}{n-k}\frac{n}{k} }{C_n}\eqpd\]
\end{lemma}
\begin{proof}
    It is well-known  (see \cite{DyckPathEnumeration1999}) that the number of Dyck paths with $k$ returns is \be \label{eq:NumberofdyckPathsWithKReturns}
    \binom{2n-k-1}{n-k}\frac{n}{k}. 
    \ee Thus, the probability that a uniformly random  Dyck path has $k$ returns is given by dividing \eqref{eq:NumberofdyckPathsWithKReturns} by the total number of Dyck paths of length $n$ which is given by the Catalan number $C_n$.
    This result applies directly to weakly increasing parking functions by way of the aforementioned bijection between Dyck paths and weakly increasing parking functions.
\end{proof}

\section{Probabilistic Parking in the Unbounded Model}\label{sec:UnboundedModel}
Let $\alpha = (\alpha_1, \alpha_2, \ldots, \alpha_n) \in \idPF$.
Since $\alpha$ is a parking function whose outcome is the identity permutation, we have that, under the classical parking protocol, car $i$ displaces
\begin{equation*}
    d_i \coloneqq i - \alpha_i
\end{equation*}
spots before parking in spot $i$. 
Now, consider the arrival of car $i$ under the unbounded stochastic model, assuming spots $1, 2, \ldots, i - 1$ are all occupied. In this case, under the unbounded model car $i$ (and thus $X_{\alpha}^i=1$) if and only if car $i$ hits position $i$ before drifting off to infinity on the line.
As it turns out, the analysis for this single-car setting reduces to the analysis of McMullen~\cite{mcmullen2020drunken} for the ``bust time'' of a gambler with some initial endowment. In particular, McMullen considers the following classical gambler's ruin set up. A gambler with an initial endowment of $k$ dollars begins to place bets in a game where, with probability $p$, the gambler wins one dollar, and with probability $q=1-p$, they lose one dollar. McMullen is interested in evaluating the function $c(b,k)$ which is the number of ways the gambler can become ruined (i.e., be left with zero dollars) after starting to gamble with $k$ dollars. McMullen showed that $c(b,k)$ satisfies the recurrence relation 
\be \label{eq:McMullenRecurrence}
c(b, k)=c(b-1, k-1)+c(b-1, k+1)\eqcom
\ee
with boundary conditions
$c(0, 0) = 1$, $c(\kappa, \kappa) = 1$,  and $c(\kappa,0)= 0$.
One can also think of this as the number of paths on $\Z^{+}$ of length $b$ starting from position $k$ which stay positive until ending at position $0$.

Translating this to our probabilistic parking process, when $\alpha$ is a parking function whose outcome is the identity permutation, car $i$ starts its random walk on $\mathbb{Z}$ in spot $\alpha_i \in [i]$, and it must displace $d_i$ spots to the right of spot $\alpha_i$ before parking in spot $i$.
In the best case scenario, in total, car $i$ displaces to left exactly $0$ times and to right exactly $d_i$ times, for a combined displacement of $d_i$.
This can happen in exactly $1$ way.
In the next best case scenario, in total, car $i$ displaces to the left exactly $1$ time and to the right exactly $d_i + 1$ times, for a combined displacement of $d_i + 2$.
This can happen in exactly $d_i$ ways.
In general, if car $i$ displaces to the left exactly $\ell \in \mathbb{N}$ times, then it must displace to the right exactly $d_i + \ell$ times for a combined displacement of $d_i + 2\ell$. 

One can see that the number of ways in which car $i$ displaces a total of $\kappa \geq 0$ times, starting from $\delta \leq \kappa$ spots to the left of spot $i$, before parking in spot $i$ is exactly characterized by the recurrence relation given in \eqref{eq:McMullenRecurrence}. We denote this number by $c^i(\kappa, \delta)$. In particular, we have that car $i$ parks in spot $i$ after taking $d_i+2\ell$ steps of a random walk starting from position $\alpha_i$ if and only if a gambler becomes ruined after $d_i+2\ell$ rounds of a game starting with $d_i$ dollars. Additionally, the probability of winning a dollar corresponds to the probability of a car taking a step to the left, which is given by $q$. Thus we are interested in the value of $c^i(d_i+2\ell, d_i)$ with the same boundary conditions as in the gambler's ruin scenario. McMullen~\cite{mcmullen2020drunken} showed that that the recurrence relation \eqref{eq:McMullenRecurrence} unravels to obtain
\begin{equation*}
    c^i(d_i + 2\ell, d_i) = \frac{d_i}{\ell + d_i}\binom{2\ell + d_i - 1}{\ell},
\end{equation*}
which is the $\ell$th term of the $d_i$th Catalan self-convolution~\cite{catalan1887sur}.
From this, we obtain the following.

\begin{theorem}
\label{thm:UnboundedProbabilityOfParking}
Let $\alpha = (\alpha_1, \alpha_2, \ldots, \alpha_n) \in \idPF$.
Under the unbounded model, $\Pr[X_\alpha^1] = 1$ and, for any $1 < i \leq n$, we have
\begin{equation}\label{eq:unboundedProbOfParking}
    \Pr\left[X_\alpha^i = 1 \big| \prod_{j=1}^{i-1}X_\alpha^{j} = 1 \right] 
    = 
    \sum_{\ell = 0}^\infty \frac{d_i}{\ell + d_i}\binom{2\ell + d_i - 1}{\ell} p^{d_i + \ell}q^{\ell}.
\end{equation}
Moreover, $\Pr\left[X_\alpha^i = 1  \big| \prod_{j=1}^{i-1}X_\alpha^{j} = 1\right] = 1$ if and only if $1/2 \leq p \leq 1$.
\end{theorem}

Note that McMullen showed that in the case where $0\leq p <\frac{1}{2}$, that the left hand side of \eqref{eq:unboundedProbOfParking} simplifies to 
$\left(\frac{p}{q}\right)^{d_i}$.
Recall that, in the unbounded model, our stopping time is defined as
\[\tau_{\alpha}^{i} \coloneqq \min\{ t \geq 0\ : Z_{t}^i = i \mbox{ or } Z_{t}^i = \infty \} \eqpd \]
Using this, we obtain the following result about the expected value.
\begin{theorem}
\label{thm:UnboundedExpectedTimeVariance}
Let $\alpha = (\alpha_1, \alpha_2, \ldots, \alpha_n) \in \idPF$.
Under the unbounded model, $E[\tau_\alpha^1] = 0$ and, for any $1 < i \leq n$ and $1/2 < p \leq 1$, we have
\begin{align*}
    \E\left[\,\tau_\alpha^i \big| \prod_{j=1}^{i-1}X_\alpha^{j} = 1\right] 
    &= 
    \frac{d_i}{p-q}
\intertext{and}
    \Var\left[\tau_\alpha^i \big| \prod_{j=1}^{i-1}X_\alpha^{j} = 1\right] 
    &= 
    \frac{4d_ip \cdot q}{(p-q)^3}.
\end{align*}
\end{theorem}

\begin{corollary}\label{cor:ResultsForAllCarsUnbounded}
Let $\alpha = (\alpha_1, \alpha_2, \ldots, \alpha_n) \in \idPF$.
Under the unbounded model, we have,
\[
  \Pr\left[X_\alpha = 1 \right] = \begin{cases}
    1 &  \mbox{ if }  \frac{1}{2}\leq p \leq 1\\
    \left(\frac{p}{q}\right)^{\sum\limits_i^nd_i} & \mbox{ if } 0\leq p <\frac{1}{2}\eqpd
  \end{cases}
\]
Furthermore, for $1/2 < p \leq 1$, we have
\begin{align*}
    \E\left[\,\tau_\alpha \big| X_{\alpha}=1\right] 
    &= 
    \frac{1}{p-q} \sum_{i=1}^n d_i
\intertext{and}
    \Var\left[\tau_\alpha | X_{\alpha}=1\right] 
    &= 
    \frac{4p\cdot q}{(p-q)^3} \sum_{i=1}^n d_i.
\end{align*}
\end{corollary}

\begin{proof}
    Let $A_i$ denote the event that the $i$th car to arrive parks in spot $i$. Note that, by our assumption that $\alpha \in \idPF$, the probability that the first $k \in [n]$ cars park is given by
\[
    \Pr\left[\cap_{i=1}^{k} A_i\right] =
         \prod_{i=1}^{k} \Pr[A_i | \cap_{j=1}^{i-1} A_j] \eqpd
\]
Thus, using this fact and Theorem ~\ref{thm:UnboundedProbabilityOfParking} gives the formula for the probability that all of the cars park. To compute the expected time it takes for the parking process to complete, we use the fact that $\tau_{\alpha}$ is such that 
\[
\E\left[\,\tau_{\alpha} \mid X_{\alpha}=1\right]
=\sum_{i=1}^n \E\left[\,\tau_{\alpha}^i \mid Z^i_{\tau^i_{\alpha}}=i \right] \eqpd
\]
This in combination with Theorem ~\ref{thm:UnboundedExpectedTimeVariance} gives the desired result.
\end{proof}
\begin{remark}
    In Corollary~\ref{cor:ResultsForAllCarsUnbounded}, several of the results make use of the \textit{total displacement} statistic (the sum of the deterministic displacement of each of the cars) showing that it is indeed central to determining the probability, expected time, and variance of the probabilistic parking process. For an overview of results related to the total displacement statistic and parking functions, see \cite{yan2015parking}.
\end{remark}

\section{Probabilistic Parking with Open Boundaries}\label{sec:OpenBoundary}
In this section, we give analogous results for the probabilistic parking process with open boundaries.

\begin{figure}
  \begin{center}
\includegraphics[scale = 0.5]{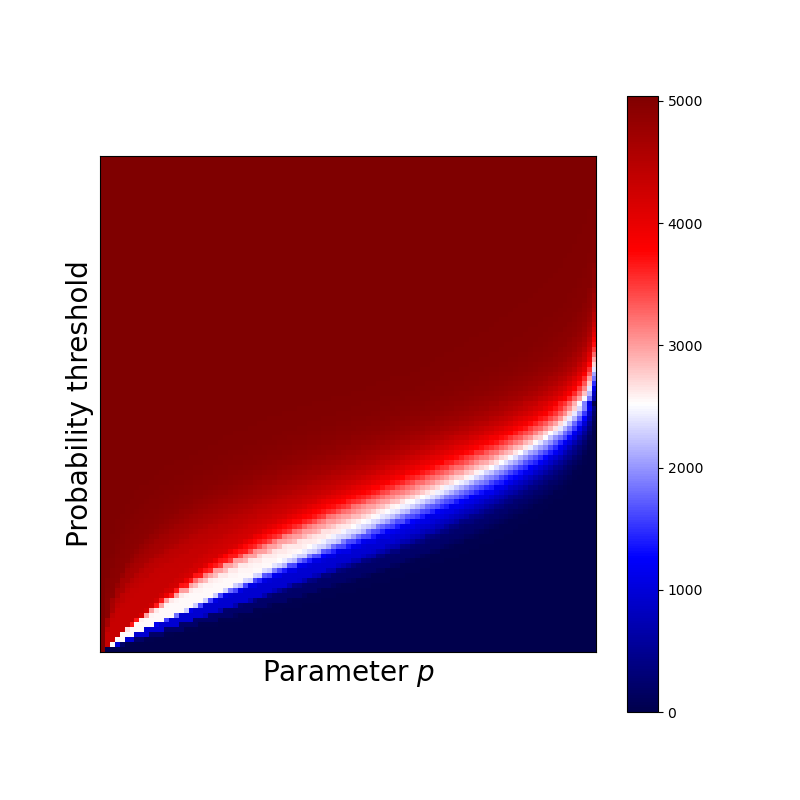}
\end{center}
    \caption{The color of each cell $(p, y)$ in the image above corresponds to the number of parking functions $\alpha \in \idPF$ where the probability that $\alpha$ parks with parameter $p$ is at most $y$, where $y$ is calculated according to Theorem~\ref{thm:FormulaswPF}. For example, when $p = 0$, the probability that $\alpha$ parks is always $0$ and the pixels are dark red for all $y$.}
    \label{fig:enter-label}
\end{figure}

\begin{theorem}\label{thm:FormulaswPF}
    Let $\alpha = (\alpha_{1}, \alpha_2, \ldots, \alpha_n) \in \idPF$.
In the probabilistic parking protocol with open boundary condition, we have that, for any $1 < i \leq n$,
\begin{equation*}
      \Pr\left[X_\alpha^i = 1 \mid \prod_{j=1}^{i-1}X_\alpha^{j} = 1\right] = 
    \begin{cases}
      \displaystyle\frac{
    \alpha_i}{i}
     &\mbox {if } \ \ p = \frac{1}{2}\\
    
         p^{i-\alpha_i}\left(\displaystyle\frac{p^{\alpha_i}-q^{\alpha_i}}{p^i-q^i} \right)\, & \mbox{otherwise}\eqpd
    \end{cases}
\end{equation*}

\end{theorem}

Thus, as a corollary we have the following.

\begin{corollary}
Let $\alpha = (\alpha_{1}, \alpha_2, \ldots, \alpha_n) \in \idPF$. Under the probabilistic parking protocol with open boundaries and initial preferences given by $\alpha$, the probability that all of the cars are able to park is given by 
\begin{equation*}
      \Pr\left[X_\alpha = 1 \right]=
    \begin{cases}
     \prod\limits_{k = 1}^{n} \frac{
    \alpha_k}{k}
     &\mbox {if } \ \ p = \frac{1}{2}\\
    
         \displaystyle\prod\limits_{k= 1}^{n} p^{k-\alpha_k}\left(\displaystyle\frac{p^{\alpha_k}-q^{\alpha_k}}{p^k-q^k} \right)\, & \mbox{otherwise}\eqpd
    \end{cases}
\end{equation*}

In particular, notice that when $p=q = \frac{1}{2}$ this is simply the product of all of the parking preferences normalized by $\frac{1}{n!}$.
\end{corollary}

\begin{proof}[Proof of Theorem \ref{thm:FormulaswPF}]
Again, let $A_i$ denote the event that the $i$th car to arrive parks in spot $i$, and recall that by our assumption that $\alpha \in \idPF$, the probability that the first $k \in [n]$ cars park is given by
\be\label{eq:intersectionOfEvents}
    \Pr\left[\cap_{i=1}^{k} A_i\right] =
         \prod_{i=1}^{k} \Pr[A_i | \cap_{j=1}^{i-1} A_j] \eqpd
\ee

Let $0 \leq w_j^i \leq 1$ be the probability that car $i$ parks before leaving the segment given that the cars in the set $\{1, 2, \ldots, i - 1\}$ park, and that it attempts to park directly in spot $0 \leq j \leq i$. In other words, $w_j^i$ is the probability car $i$ parks assuming $\alpha_i =j$.
Then, $w_i^i = 1$ and $w_0^i = 0$, and
more generally, for any $0 < j < i$ we have
\begin{equation}
    w_j^i = p \cdot w_{j+1}^i + q\cdot w_{j-1}^i.
\end{equation}
Solving this system of equations when $p \neq \frac{1}{2}$ (for the technique used refer to: \cite[Chapter~12.2]{grinstead2006grinstead}), assuming that car $i$ starts in position $s \in [n]$ where $ 1 \leq  s \leq i$,  we have that
\be\label{eq:probabilityofParkingOpen}
 w_s^i  = \frac{1-\left(\frac{q}{p}\right)^{s}}{1-\left(\frac{q}{p}\right)^i} = \frac{p^{i+s}-p^iq^s}{p^{i+s}-p^{s}q^{i}} 
 \eqpd
\ee

Recall that we use $Z_t^i$ to denote the random variable keeping track of the position of car $i$ at time $t$. Furthermore, recall that we have defined the stopping time
\[
\tau_{\alpha}^{i} \coloneqq \min\{ t \ : Z_{t}^i = i \mbox{ or } Z_{t}^i = 0 \}
\]
for each car $i \in [n]$. Therefore, by definition, we have that
\[
\Pr[A_i | \cap_{j=1}^{i-1} A_j] = \Pr[Z_{\tau_{\alpha}^{i}}^i = i\  |  \ \cap_{j=1}^{i-1} A_j] =   w^i_{\alpha_i}\eqpd
\]
And so, using \eqref{eq:intersectionOfEvents} and \eqref{eq:probabilityofParkingOpen}, the probability that all cars are able to park is given by 

\[
\Pr[X_\alpha = 1] = \prod\limits_{i = 1}^{n} w^i_{\alpha_i} = \prod\limits_{i = 1}^{n} \frac{p^{i+\alpha_i}-p^iq^{\alpha_i}}{p^{i+\alpha_i}-p^{\alpha_i}q^{i}} =\prod\limits_{i = 1}^{n}p^{i-\alpha_i}\left(\frac{p^{\alpha_i}-q^{\alpha_i}}{p^i-q^i}\right) \eqpd
\]
The case of $p=\frac{1}{2}$ is simpler, and can be handled similarly. 
\end{proof}

\subsection{Expected Time to Park}\label{sec:ExpectedTimeToPark}
Recall that $Z_t^i$ is the random variable keeping track of the position of car $i \in [n]$ at time $t \in [n] \cup \{0\}$ in the probabilistic parking protocol. Let 
\[\tau_{\alpha}^{i} \coloneqq \min\{ t \geq 0\ : Z_{t}^i = i \mbox{ or } Z_{t}^i = 0 \} \]

denote the random stopping time which is the first time that car $i$ hits position $i$ or position $0$.
Recall that we let $\tau_\alpha$ denote the time it takes for all cars to park. In other words, $\tau_{\alpha}$ is the stopping time such that  
\[ 
\E\left[\tau_{\alpha} \mid X_{\alpha}=1\right]
=\sum_{i=1}^n \E\left[\tau_{\alpha}^i \mid Z^i_{\tau^i_{\alpha}}=i \right]\eqpd
\]
Recall that, by the dynamics of the process, in order for all of the cars to successfully park we must have that each car $i$ parks on its ``first attempt''. In other words, if a car encounters either position $n+1$ or position $0$ before position $i$, then the entire parking process was unsuccessful. Thus, we have the following theorem giving the expected time for a single car to park, given that all previous cars park. 
\begin{theorem} \label{thm:ExpectedTimeOpenBoundaries}
Given  $\alpha = (
\alpha_1, \alpha_2, \ldots, \alpha_n) \in \idPF$ the expected time for a single car to park under the probabilistic parking protocol with open boundaries, given that all the previous cars have parked, is given by
\[
\E\left[\tau_{\alpha}^i ~|~\prod_{j=1}^{i}X_\alpha^{j} = 1 \right] =
 \begin{cases}
     \frac{1}{3}(i^2 - \alpha_i^2) 
     &\mbox {if } \ \ p =\frac{1}{2}\\
      \frac{1}{p-q}\left( \frac{
    i\left(1 + \left(\frac{q}{p}\right)^i\right)
}{ \left(1 - \left(\frac{q}{p}\right)^i\right)} - \frac{\alpha_i \left(1 + \left(\frac{q}{p}\right)^{\alpha_i}\right)}{\left(1 - \left(\frac{q}{p}\right)^{\alpha_i}\right)}\right)  \, & \mbox{otherwise}\eqpd
         
    \end{cases}
\]
\end{theorem}
The following corollary gives the expected time for all of the cars to park. 
\begin{corollary}
Given  $\alpha = (
\alpha_1, \alpha_2, \ldots, \alpha_n) \in \idPF$ the expected time for all of the cars to park under the probabilistic parking protocol with open boundaries is given by
\[ 
\E\left[\,\tau_{\alpha}~|~ X_\alpha =1 \right] =
 \begin{cases}
     \frac{1}{3}\sum\limits_{i=1}^n(i^2 - \alpha_i^2) =\frac{2n^3+3n^2+n}{18}-\frac{1}{3}\sum\limits_{i}^n\alpha_i^2
     &\mbox {if } \ \ p = \frac{1}{2}\\
      \frac{1}{p-q}\sum\limits_{i=1}^n \frac{
    i\left(1 + \left(\frac{q}{p}\right)^i\right)
}{ \left(1 - \left(\frac{q}{p}\right)^i\right)} - \frac{\alpha_i \left(1 + \left(\frac{q}{p}\right)^{\alpha_i}\right)}{\left(1 - \left(\frac{q}{p}\right)^{\alpha_i}\right)}  \, & \mbox{otherwise}\eqpd
\end{cases}
\]
\end{corollary}
\begin{proof}[Proof of Theorem \ref{thm:ExpectedTimeOpenBoundaries}]
Let 
$g^{i}_s \coloneqq \E_s\left[ \tau_{\alpha}^{i}\mid \prod_{j=1}^{i-1}X_\alpha^{j} = 1 \right]$ where we use $\E_s$ to denote the expectation of the process starting from $s$. In other words $g^i_s$ is the expected time it takes for car $i$ to successfully park when starting from position $s$, and thus, $g^i_0 = g^i_{i} = 0$. Furthermore, letting $\gamma$ be a path from $s$ to $i$, we have that
\begin{align*}
g^{i}_s &= \sum_{\ell =1}^{\infty} \ell \Pr\left[ |\gamma| = \ell \big| \left\{Z_{\tau_{\alpha}^{i}}^i = i,  Z_{0}^i=s \right\} \right]\\
&= \sum_{\ell=1}^{\infty} \ell\left(w^i_{s}\right)^{-1}\Pr\left[ \left\{|\gamma| = \ell\right\}\cap \left\{Z_{\tau_{\alpha}^{i}}^i = i,  Z_{0}^i=s \right\}\right]\\
&= \left(w^i_{s}\right)^{-1}\sum_{\ell =1}^{\infty}\sum_{\substack{\substack{\substack{\substack{|\gamma| = \ell \\ \gamma_0 = s} \\ \gamma_1 = s+1}
\\ \vdots } \\ \gamma_{\ell} = i}}^{\infty} \ell \Pr\left[\gamma \cap \left\{Z_{\tau_{\alpha}^{i}}^i = i,  Z_{0}^i=s \right\}\right] + \sum_{\substack{\substack{\substack{\substack{|\gamma| = \ell\\ \gamma_0 = s} \\ \gamma_1 = s-1}
\\ \vdots } \\ \gamma_{\ell} = i}}^{\infty} \ell \Pr\left[ \gamma \cap \left\{Z_{\tau_{\alpha}^{i}}^i = i,  Z_{0}^i=s \right\}\right]\eqpd\\
\end{align*}
The last equality separates the sum into two, the first being the sum over paths which take their first step right, and the second sum is over the paths that take their first step left.
For the leftmost summation, we can make this first rightward step explicit and further decompose the sum into
\begin{align*}
&  \sum_{\ell=1}^{\infty} \sum_{\substack{\substack{\substack{\substack{|\gamma| = \ell \\ \gamma_0 = s} \\ \gamma_1 = s+1}
\\ \vdots } \\ \gamma_{\ell} = i}}^{\infty} \left((\ell -1)\Pr\left[ \gamma  \cap \left\{Z_{\tau_{\alpha}^{i}}^i = i,  Z_{0}^i=s \right\}\right] +  \Pr\left[ \gamma \cap \left\{Z_{\tau_{\alpha}^{i}}^i = i,  Z_{0}^i=s \right\}\right]\right)\nonumber\\
&= \sum_{\ell=1}^{\infty} p \sum_{\substack{\substack{\substack{|\beta| = \ell-1 \\ \beta = s+1}
\\ \vdots } \\ \beta_{\ell} = i}}^{\infty} \left((\ell -1)\Pr\left[ \beta \cap \left\{Z_{\tau_{\alpha}^{i}}^i = i,  Z_{0}^i=s+1 \right\}\right] +  \Pr\left[  \beta \cap \left\{Z_{\tau_{\alpha}^{i}}^i = i,  Z_{0}^i=s+1\right\}\right]\right)\nonumber\\
&= \sum_{\ell=1}^{\infty} p \sum_{\substack{\substack{\substack{|\beta| = \ell-1 \\ \beta = s+1}
\\ \vdots } \\ \beta_{\ell} = i}}^{\infty} (\ell -1)\Pr\left[ \beta \cap \left\{Z_{\tau_{\alpha}^{i}}^i = i,  Z_{0}^i=s+1 \right\}\right] + \sum_{\substack{\substack{\substack{|\beta| = \ell-1 \\ \beta = s+1}
\\ \vdots } \\ \beta_{\ell} = i}}^{\infty} \Pr\left[  \beta \cap \left\{Z_{\tau_{\alpha}^{i}}^i = i,  Z_{0}^i=s+1\right\}\right]\nonumber\\
\end{align*}

By the definition of $w^i_s$ and $g^i_s$ this last line is just 
\[
pw^i_{s+1}\left(g_{s+1}^i +1\right) \eqpd
\]

Repeating the above argument gives an analogous result for the case in which the path starts with a step to the left. Therefore, we have the recurrence relation

\[
\begin{split}
    g_{s}^i&=\left(w_{s}\right)^{-1}pw_{s+1}\left(g_{s+1}^i +1\right)  + \left(w_{s}\right)^{-1}qw_{s-1}\left(g_{s-1}^i +1\right) \\
    &=\frac{pw_{s+1}\left(g_{s+1}^i\right)}{w_{s} }
+ \frac{qw_{s-1}\left(g_{s-1}^i\right)}{w_{s} }
+\frac{pw_{s+1} +qw_{s-1} }{w_{s} }\\
&=\frac{pw_{s+1}\left(g_{s+1}^i\right)}{w_{s} }
+ \frac{qw_{s-1}\left(g_{s-1}^i\right)}{w_{s} }
+1
\end{split}
\]
where the equality holds by the recursive  definition of $w^i_s$.
In  Appendix~\ref{sec:appendix}, we establish Proposition~\ref{prop:SolutionToLinearSystemAppendix}, which shows that the linear system has solution 

\[
g_s^i= \frac{
    i\left(1 + \left(\frac{q}{p}\right)^i\right)
}{(p-q) \left(1 - \left(\frac{q}{p}\right)^i\right)} - \frac{s \left(1 + \left(\frac{q}{p}\right)^s\right)}{(p-q) \left(1 - \left(\frac{q}{p}\right)^s\right)}\eqpd
\]
 Finally, replacing $s$ with $\alpha_i$ for each car $i$, and taking the sum over all cars $i \in [n]$ yields the desired result.
\end{proof}

\begin{remark}
The set of parking functions that are weakly increasing are in bijection with those that are weakly decreasing via the involution that takes $i \mapsto n-i+1$, with the added condition that we reverse the order of entrance into the street. That is, the cars would now enter from the right, and perform the parking protocol moving toward spot 1.
This involution and change of parking direction in the parking protocol also give a bijection between the set of parking functions whose outcome is the identity permutation and the set of (right to left) parking functions whose outcome is the 
longest word, namely the permutation $n(n-1)(n-2)\cdots 321$ written in one-line notation. We denote the set of all such parking functions by $\PF_n(n\cdots 321)$ and note that they are characterized as the set of parking functions $\beta \in [n]^n$ such that the $i$th entry $\beta_i$ is at least $n-i+1$ for all $i \in [n]$.
By way of this bijection, all of our previous arguments and results for the probabilistic parking protocols (both in the open boundary and the unbounded models) work analogously when starting with the set $\PF_n(n\cdots 321)$,  of parking functions with outcome $n\cdots321$.  
\end{remark}

\subsection{Combinatorial Derivation of Expected Time to Park}
In this section, we start the process with preference list $\alpha \in \idPF$ and give a combinatorial derivation of the expected time it takes for all cars to successfully park under the probabilistic parking protocol with open boundaries. 
Since $\alpha\in\idPF$, we know that car $i$ must park in spot $i$. In order to calculate the expected time a single car takes to park, we enumerate the number of paths that car $i$ can take from spot $\alpha_i$ to spot $i$ in a fixed amount of time.

First, we consider the number of paths from spot $1$ to spot $i$ in $t$ steps. Following \cite{zigzag}, we let $a_{j,k}$ be the number of paths to spot $k+1$ using $j$ steps to the left. 
We can think of $k$ as the net number of spots that the car has moved to the right, so $k+2j$ is the total number of steps the paths enumerated by $a_{j,k}$ requires. 
There is one path to any spot using only right steps, so $a_{0,k}=1$. 
We also have that, analogously to \eqref{eq:McMullenRecurrence}, the number of paths can be determined by the recurrence relation $a_{j,k}=a_{j,k-1}+a_{j-1,k+1}$. This is because a path to spot $s$ can be obtained from either a path landing at spot $s-1$ followed by one right step, or a path landing in spot $s+1$ followed by one left step. Unlike in \eqref{eq:McMullenRecurrence} however, a path in the open boundary case cannot end in spot $0$. Lastly, since car $i$ always parks in spot $i$ before reaching spot $i+1$, we have that the number of paths landing in spot $k+1$ in $k+2j$ steps is the sequence $a_{j,k}$ from \cite{zigzag} given by

\begin{align}\label{recu}
a_{i,j} = 
\begin{cases}
    1 & \mbox{ if } j=0 \\
    a_{j-1,1} & \mbox { if } k=0, \ 1 \leq j \\
    a_{j,k-1} + a_{j-1,k+1} & \text{ if } 1 \leq k \leq i, \ 1 \leq j \\
    a_{j,\ell} & \mbox{ if } \ell=i+1, \ 1 \leq j.
\end{cases}
\end{align}
For any $i$, L\'aszl\'o N\'emeth and L\'aszl\'o Szalay refer to the sequence $(a_{j,k})_{j=0}^\infty$ as $A_k^{(i)}$. Using this notation, the number of paths that car $i$ has starting at spot 1, moving between spot $1$ and $i-1$ in the first $k+2j -1$ steps of the walk and ending at spot $i$ at step $k+2j$ step is the $j$th element in the sequence $A_{i+2}^{(i+2)}$. 
We use this to find the expected time it takes for car $i$ to travel from spot 1 to spot $i$, given that car $i$ eventually does park which we again denote by $g_1^i$. We remark that the approach in this section is very similar to that used by McMullen \cite{mcmullen2020drunken}, however in their case the terms of the sequence $A_k^{(i)}$ are Catalan convolutions because of the different boundary condition.  Using the  recurrence relation in \eqref{recu}, we have the following theorem.
\begin{theorem}\label{thm:ExVal}
Given a parking function $\alpha = (
\alpha_1, \alpha_2, \ldots, \alpha_n) \in \idPF$, the expected time for car $i$ to park in spot $i$ starting from spot 1 ($\alpha_i =1$) provided it parks and all the previous cars have parked under the probabilistic parking protocol with open boundaries is given by 
 \[g^i_1=\E_1\left[\tau_\alpha^i~|~\prod_{j=1}^{i}X_\alpha^{j} = 1 \right]=\sum_{k=0}^\infty \frac{ A_{i-3}^{(i-3)}(k) p^{k+i-1} q^{k} (2k+i-1)  }{w^i_1}.\]
\end{theorem}
\begin{proof}
We have remarked that the $j$th element of $A_{i-3}^{(i-3)}$ is the number of paths from 1 to $i$ with $j$ steps to the left. Since we know any path from $1$ to $i$ has $i-1$ more right steps than left steps, we have $j + j + i-1$ total steps in the walk. The probability of such a path occurring from the set of any path is $q^j$ times $p^{j+i-1}$ for the left and right steps, respectively. Where we again use $w_1^i$ to denote the probability of car $i$ parking in spot $i$ before leaving the segment. Thus, since we only want to sample from the functions that park we multiply $p^{k+i-1}q^k$ by $1/w_1^i$ for the probability of selecting a path with $j$ left steps and $2j+i-1$ total steps, from those which never go to the left of spot 1. 
\end{proof}
By a similar argument as in the proof of Theorem \ref{thm:ExVal}, we also have the following.
\begin{corollary}
Given a parking function $\alpha = (
\alpha_1, \alpha_2, \ldots, \alpha_n) \in \idPF$, the expected time for car $i$ to park in spot $i$ starting from spot $i-1$ provided it parks and all the previous cars have parked
under the probabilistic parking protocol with open boundaries is given by 
    $$\E_{i-1}\left[\tau_\alpha^i~|~\prod_{j=1}^{i}X_\alpha^{j} = 1 \right] = \sum_{k=0}^\infty \frac{ A_{0}^{(i-3)}(k) p^{k} q^{k+1} (2k+1)  }{w^i_{i-1}}\eqpd$$
\end{corollary}

These results match those in Subsection \ref{sec:ExpectedTimeToPark}, yet we illustrate them again to provide an alternate proof method we find to be of independent interest. 

\section{The Parking Process is Negatively Correlated}\label{sec:NegativeCorrelation}

In this section, we begin by proving that the indicator random variables $X^i_{\alpha}$, which are equal to one if car $i$ parks during the probabilistic parking protocol with open boundaries starting with preference list $\alpha$, and zero otherwise, are \textit{pairwise negatively correlated}. 

\begin{definition}[Pairwise Negative Correlation]\label{def:PairwiseNegativeCorrelation}
A collection $\{X_1, \ldots, X_n\}$ of random variables are pairwise negatively correlated if, for all $1 \leq i<j \leq n$,
\be\label{eq:pairwiseNegCor}
\mathbb{E}\left[X_i X_j\right] \leq \mathbb{E}\left[X_i\right] \mathbb{E}\left[X_j\right]\eqpd
\ee
\end{definition}

We also show that the random variables are \textit{negatively correlated}, more generally.

\begin{definition}[Negative Correlation] \label{def:NegativeCorrelation}
    A collection $\left\{X_1, \ldots, X_n\right\}$ of random variables are negatively correlated if, for any subset $S \subseteq[n]$, we have
\[
\mathbb{E}\left[\prod_{i \in S} X_i\right] \leq \prod_{i \in S} \mathbb{E}\left[X_i\right] .
\]
Note that if $\left\{X_1, \ldots, X_n\right\}$ are independent, then the inequality holds at equality.
\end{definition}

Going forward, we will abuse the definition slightly and say that the parking process is negatively correlated. 
Notice that, if a set of indicator random variables $\{X_1, \dots, X_n$\} are negatively correlated, definitions \ref{def:PairwiseNegativeCorrelation} and \ref{def:NegativeCorrelation} become statements about joint probabilities rather than expectations. Furthermore,  using the definition of conditional probability, the property of being negatively correlated implies that one event occurring makes the others less likely to occur. 

In the case of our parking protocol with open boundaries, this is exactly as one might expect. Once a car parks, there are fewer spots in the parking lot and so the probability a car is able to park before leaving the segment should only remain the same or decrease. Note that in the context of the unbounded model on the other hand, the indicator random variables (defined analogously to $X^i_{\alpha}$) are not negatively correlated. This is because, a car parks under the unbounded model starting with $\alpha \in \idPF$ if and only if it is able to hit position $i$ before drifting off to infinity. However in the case that car $i$ drifts off to infinity, none of the cars in the set $[i+1, n]$ are able to begin the parking process, and thus park with probability zero. We will also give a proof of this fact at the end of the section.

To formalize our intuition for the model with open boundaries, we first start by proving that the random variables are pairwise negatively correlated for any $\alpha \in [n]^n$, and thus they are pairwise negatively correlated when starting the protocol with $\alpha \in \idPF$.
\begin{lemma}[Pairwise Negative Correlation of Parking]\label{lem:PairwiseNegCor}
   For any $\alpha \in [n]^n$, the sequence of indicator random variables $\{X^i_{\alpha}\}_{i \in [n]}$ defined with respect to the probabilistic parking protocol with open boundaries are pairwise negatively correlated.
\end{lemma}

\begin{proof}
Consider two indices $i,j  \in [n]$ with $i<j$. First notice that, while the probability of car $i$ parking is not influenced by car $j$ whenever $i <j$, the events are still correlated. For the left hand side of \eqref{eq:pairwiseNegCor}, we have 
\[
\Pr[X^i_\alpha X^j_\alpha=1] = \Pr[X^i_\alpha =1] \Pr[X^j_\alpha =1 | X^i_\alpha =1]\eqpd
\]
Thus, what is left to show is that 
\be\label{eq:questionmarkIneq}
\begin{split}
\Pr[X^j_\alpha =1 | X^i_\alpha =1] &\stackrel{?}{\leq} \Pr[X^j_\alpha =1]\eqpd \\
\end{split}
\ee

Let $\operatorname{Out}(i)$ be the random variable denoting the parking spot in which car $i$ parks as a result of the process, or $0$ if car $i$ does not manage to park. Notice that 
\[
\begin{split}
    \Pr\left[X_j^\alpha=1 \mid X_i^\alpha=1\right]&=\sum_{k=1}^n \Pr\left[X_j^\alpha=1 \mid \operatorname{Out}(i)=k\right] \cdot \Pr\left[\operatorname{Out}(i)=k \mid X_i^\alpha=1\right]\eqcom\\
\end{split}
\]
where for each $k$, 

\[\Pr\left[X_j^\alpha=1 \mid \operatorname{Out}(i)=k\right] \leq \Pr\left[X_j^\alpha=1, \operatorname{Out}(j) \neq k\right] \leq \Pr\left[X_j^\alpha=1\right]\eqpd
\]

Thus, we have 
  \[ 
  \begin{split}
  \Pr\left[X_j^\alpha=1 \mid X_i^\alpha=1\right]&\leq \Pr\left[X_j^\alpha=1\right]\sum_{k=1}^n \Pr\left[\operatorname{Out}(i)=k \mid X_i^\alpha=1\right]
\\
  &=\Pr\left[X_j^\alpha=1\right]
\end{split}
  \]

this shows that \eqref{eq:questionmarkIneq}, is indeed an inequality, and concludes the proof. 

\end{proof}

In the previous arguments, we rely on the fact that knowing that a later car in the parking process has been able to park does not tell us any information as to whether or not the earlier cars in the sequence were able to park. This same concept underlies the more general case which we consider next.
\begin{theorem}[General Negative Correlation]\label{thm:GeneralNegativeCorrelation}
    For any $\alpha \in [n]^n$, the sequence of random variables $\{X^i_{\alpha}\}_{i \in [n]}$ defined with respect to the probabilistic parking protocol with open boundaries are negatively correlated.
\end{theorem}

\begin{proof}
Let $S \subset [n]$ be an arbitrary subset of the cars of size $|S|=s$. Furthermore, let $\bar{S}$ denote the subset $S$ which is ordered according to the labels of the cars such that\\$\bar{S}\coloneqq\{i_1, \dots i_s\}$ where $i_1 < \dots < i_s$. We prove the theorem by showing that 
\be\label{eq:generalizedIneq}
\prod\limits_{m =i_1}^{i_s}\Pr\left[X^m_{\alpha}=1 \mid \prod\limits_{j = i_1}^{m-1}  X^j_{\alpha}=1\right] \stackrel{?}{\leq} 
\prod\limits_{i \in S}\Pr[ X^i_{\alpha}=1] 
\ee

as the left hand side of the above inequality is precisely equal to $\Pr\left[ \prod\limits_{i \in S}X^i_{\alpha}=1\right] $ by the definition of conditional probability and the probabilistic parking protocol. 
Thus, we proceed by induction on the size of the set $S$. 
For the base case, when $S$ is of size 1, the desired inequality holds directly. 
In addition, for $S$ containing two cars, Theorem~\ref{lem:PairwiseNegCor} proves the inequality holds. 
Now assume for induction that the result holds for $|S| = k$. 
We now show the inequality holds for $|S|=k+1$. 
Again, let $\bar{S}\coloneqq \{i_1, \dots, i_{k+1}\}$ where $i_1 < \dots < i_{k+1}$. 
Thus, using the induction hypothesis, we have
\[
\begin{split}
\prod\limits_{m =i_1}^{i_{k+1}}\Pr\left[X^m_{\alpha}=1 \mid \prod\limits_{j = i_1 }^{m-1}  X^j_{\alpha}=1\right] & \leq  \Pr\left[X^{i_{k+1}}_{\alpha}=1 \mid \prod\limits_{j = i_1 }^{i_{k}}  X^j_{\alpha}=1\right]  \prod\limits_{j= i_1}^{i_k}\Pr\left[X^j_{\alpha}=1\right] \\
\end{split}
\]
and what is left to show is that 

\[
\begin{split}
 \Pr\left[X^{i_{k+1}}_{\alpha}=1 \mid \prod\limits_{j = i_1 }^{i_{k}}  X^j_{\alpha}=1\right]  \leq \Pr\left[X^{i_{k+1}}_{\alpha}=1\right].\\
\end{split}
\]
Notice that we can partition the state space, based only on which spots are occupied when car $j$ attempts to park. Thus, let $O_K$ denote the event that the spots in the set $K \subseteq [n]$ are occupied. We have
\[\Pr\left[X_{i_{k+1}}^\alpha=1 \mid \prod_{m=1}^k X_{i_m}^\alpha=1\right]=\sum\limits_{\substack{K\subseteq [n]:\\|K| =k}}\Pr\left[X_{i_{k+1}}^\alpha=1 \mid O_K\right] \cdot \Pr\left[O_K\mid \prod_{m=1}^k X_{i_m}^\alpha=1\right] \eqcom
\]
where one can see that 
\[
\Pr\left[X_{i_{k+1}}^\alpha=1 \mid O_K\right] \leq \operatorname{Pr}\left[X_{i_{k+1}}^\alpha=1\right]
\]
holds. Thus we have, 

\[
\begin{gathered}\operatorname{Pr}\left[X_{i_{k+1}}^\alpha=1 \mid \prod_{m=1}^k X_{i_m}^\alpha=1\right] \leq \operatorname{Pr}\left[X_{i_{k+1}}^\alpha=1\right] \sum\limits_{\substack{K\subseteq [n]:\\|K| =k}} \operatorname{Pr}\left[O_K \mid \prod_{m=1}^k X_{i_m}^\alpha=1\right] \\ =\operatorname{Pr}\left[X_{i_{k+1}}^\alpha=1\right]\end{gathered}
\]

which yields the desired result. 
\end{proof}
We conclude by showing that the unbounded model does not have this negative correlation property. 
\begin{lemma}[The Unbounded Model is Not Negatively Correlated]\label{lem:UnboundedNotNegativelyCorrelated}
    Let $\alpha \in [n]^n$. The sequence of random variables $\{X^i_{\alpha}\}_{i \in [n]}$ defined with respect to the unbounded probabilistic parking protocol are \textbf{not}
    negatively correlated.
\end{lemma}

\begin{proof}
Let $i, j \in [n]$, and assume without loss of generality that $i <j$. Furthermore, assume that $i \neq 1$. For the indicator random variables $\{X^i_{\alpha}\}_{i \in [n]}$ to be negatively correlated under the unbounded model, it must be the case that 

\[
\Pr[X_{\alpha}^i X_{\alpha}^j=1] \leq \Pr[X_{\alpha}^i=1] \Pr[X_{\alpha}^j=1] 
\]
which implies that 
\be\label{eq:UnboundedNotNegCorrelated1}
\Pr[X_{\alpha}^j =1 \mid X_{\alpha}^i=1] \leq \Pr[X_{\alpha}^j=1] \eqpd
\ee
But 
\be\label{eq:UnboundedNotNegCorrelated}
\Pr[X_{\alpha}^j=1] = \Pr[X_{\alpha}^j =1\mid X_{\alpha}^i=1]\Pr[X_{\alpha}^i=1] +\Pr[X_{\alpha}^j  =1\mid X_{\alpha}^i=0]\Pr[X_{\alpha}^i=0]  \eqpd
\ee
Because car $j$ can only park if car $i$ has not drifted off to infinity, the second term in the right hand side of \eqref{eq:UnboundedNotNegCorrelated} must be equal to zero. 
Thus, 
\[
\Pr[X_{\alpha}^j=1] = \Pr[X_{\alpha}^j =1\mid X_{\alpha}^i=1]\Pr[X_{\alpha}^i=1]\eqpd
\]
Substituting this into the right hand side gives \eqref{eq:UnboundedNotNegCorrelated1}
\[
\Pr[X_{\alpha}^j =1 \mid X_{\alpha}^i=1]  \stackrel{?}{\leq}  \Pr[X_{\alpha}^j =1\mid X_{\alpha}^i=1]\Pr[X_{\alpha}^i=1] \eqcom
\]
which implies that 
\[
1 \stackrel{?}{\leq}  \Pr[X_{\alpha}^i=1]\eqpd
\]
Furthermore, by Theorem \ref{thm:UnboundedProbabilityOfParking} we know that the probability that any particular car (which is not the first car) parks with probability 1, depends on the parameter $p$. Thus, it is not true that this inequality always holds, and so the probabilistic parking process with the unbounded boundary condition, is not negatively correlated in general.
\end{proof}
\subsection{Consequences of Negative Correlation}

\begin{theorem}
Let $N_{\alpha}$ denote the number of cars which are successfully able to park during the process, i.e.~$N_{\alpha} = \sum\limits^n_{i=1}X_{\alpha}^i$. 
Then, for $0\leq \delta\leq 1$, we have
\[
\begin{split}
\operatorname{Pr}\left[N_{\alpha} \geq(1+\delta) \sum\limits_{i=1}^n\Pr\left[X_{\alpha}^i =1\right]\right] &\leq\left(\frac{e^\delta}{(1+\delta)^{1+\delta}}\right)^{\sum\limits_{i=1}^n\Pr\left[X_{\alpha}^i =1\right]} \\
\end{split}
\]
and 
\[
\begin{split}
\operatorname{Pr}\left[N_{\alpha} \leq(1-\delta) \sum\limits_{i=1}^n\Pr\left[X_{\alpha}^i =1\right]\right] &\leq\left(\frac{e^\delta}{(1-\delta)^{1-\delta}}\right)^{\sum\limits_{i=1}^n\Pr\left[X_{\alpha}^i =1\right]}\eqpd\\
\end{split}
\]
\end{theorem}
\begin{proof}
    The proof follows by adapting the classical proof of Hoeffding's inequality by substituting the assumption of negative correlation to an obtain the inequality \[\mathbb{E}\left[e^{\sum_i s X^{i}_{\alpha}}\right]=\mathbb{E}\left[\prod_i e^{s X^i_{\alpha}}\right] \leq \prod_i \mathbb{E}\left[e^{s X^{i}_{\alpha}}\right]\]
    in place of the assumption of independence (which would give an equality). The lower tail bound, on the other hand, follows by making use of the fact that the random variables $Y^i_{\alpha} \coloneqq 1-X_{\alpha}^i$ are also negatively correlated random variables, from which a similar proof yields the desired result.
\end{proof}

\section{Future work and Open Questions}\label{sec:future work}

There are several questions worth pursuing related the probabilistic parking process described in Section~\ref{sec:IntroducingTheModels}. 
We give some examples below.

\begin{itemize}
\item \textbf{Statistics of the Outcome.}
In general, it would be interesting to better understand statistics of the outcome of the probabilistic parking process for general $\alpha \in [n]^n$.
For example, we know that, if we remove the assumption that $\alpha \in \idPF$, at the end of the probabilistic parking process the parking lot could possibly have ``holes,'' or spots where no car has parked. 
Suppose that we have run the probabilistic parking protocol with preference list $\alpha \in [n]^n$, and define $B^{\alpha}$ to be a tuple of length $n$ where $B^{\alpha}_i = 1$ if spot $i$ is filled at the end of the parking process and $0$ otherwise. 
What is the probability that $B^{\alpha} =\beta$ for some $\beta \in \{0,1\}^n$? 
This line of research is closely related to the work of Varin, and thus the techniques in \cite{VarinGolf} might be useful.  
However, because we also insist on distinguishing between the cars throughout the process, one could study the probability that a specific car, say car $i$, parks in spot $j$ for $i,j \in [n]$, with or without the assumption that all of the cars are able to successfully park.
If all the cars successfully park, we could ask what is the probability that a specific permutation $\pi \in S_n$ is the outcome of the parking process started with a preference list $\alpha \in [n]^n$. 
    \item \textbf{Different Boundary Conditions.} 
    Two boundary conditions which we did not consider in this paper, but may be of interest are \textbf{Periodic Boundary Conditions}, i.e., running the probabilistic parking protocol on the circle $\Z/ n\Z$, and \textbf{Closed Boundaries}, i.e., the setting in which cars are free to drift around the finite set $[n]$, and     if a car attempts to exit this set at a given boundary, it instead remains in the same position and the move is canceled. 
    In the periodic boundary case, any preference list $\alpha \in [n]^n$ allows all of the cars to park with probability one (the same is true for reasonable choice of $p$ in the case of closed boundaries), and so determining the probability that all of the cars park is known. 
    However, it is of independent interest to consider the total time the parking process would require to stabilize under these boundary conditions. 
    Determining the distribution of the outcome permutation at the end of the process is also of interest. 
    For example, it is unknown what is the most likely parking outcome. 
    In the case of parking on the cycle, a similar gambler's ruin type argument should be useful in computing the expected time for the parking protocol to complete.
    \item \textbf{More Cars than Spots.} 
    Under the appropriate assumptions, if we were to run the probabilistic parking process with ``enough'' cars, eventually the parking lot should become completely full. 
    Under these assumptions one could consider the expected time it would take for this to occur, and the number of cars one would need to run through the process to reach this absorbing state. 
    In other words, it would be interesting to prove the existence of a phase transition dependent on the number of cars that run through the system.
    \item \textbf{Other Probabilistic Parking Protocols.} As mentioned in the introduction, there has been fairly extensive work studying parking on random objects such as random trees of various types, and frozen Erd\H{o}s R\'{e}nyi graphs. Comparatively few models of parking have been introduced which randomize the parking procedure itself, however. (For some other examples see \cite{DamronLyu1, DamronLyu2,durmic2022probabilistic,nadeau:hal-04262134,Scott, tian2021generalizing, VarinGolf}.) In many of these models, including the model discussed in this paper, the motion of the cars is characterized by random walks. In particular, the present model has deep connections to the interacting particle system IDLA. Thus it seems interesting to derive a new probabilistic parking protocol which relates to (or perhaps is independent of) other more classical particle models or stochastic processes.
\end{itemize}

\appendix
\section{Supplementary Results}\label{sec:appendix}
\begin{proposition}[Identity used in Lemma~\ref{lem:expectationOfAlphaSubNInWIPF}]\label{eq:BinomialIdentity1}
For any $n > 0$, the following equality holds:
    \[\sum_{j=0}^{n-1} j\binom{j+n-1}{j}=\frac{(n-1) n\binom{2 n-1}{n-1}}{n+1}.\] 
\end{proposition}
\begin{proof}
By some arithmetic we have that
\begin{align*}
    \sum_{j = 0}^{n-1} j \binom{j+n-1}{j} &= \sum_{j = 0}^{n-1} j \binom{j+n-1}{n-1}\\
    &= \sum_{j = 0}^{n-1} j \frac{n}{j+n}\binom{j+n}{n}\\
    &= n\sum_{j=0}^{n-1}\frac{(j+n-1)!}{(j-1)! \ n!}\\
    &= n\sum_{k = 0}^{n-2}\binom{k+n}{k}\\
    &= n \binom{2n-1}{n-2} \hfill \ \ \ \ \text{(Parallel summation)}\\
    &= \frac{(n-1) n\binom{2 n-1}{n-1}}{n+1}.\qedhere
\end{align*}
\end{proof}

\begin{proposition}[Identity used in Lemma~\ref{lem:expectationOfAlphaSubNInWIPF}]\label{eq:BinomialIdentity2}
    If $n > 0$, then
    \[\sum_{j=0}^{n - 1} j^2 \binom{j+n-1}{j} = \frac{(n-1)n^3}{(n+1)(n+2)} \binom{2n-1}{n-1}.\] 
\end{proposition}
\begin{proof}
For any $n>0$ we have that 
\begin{align*}
    &\sum_{j = 0}^{n - 1} j^2 \binom{j+n-1}{j} = n\sum_{j = 0}^{n-1}j\binom{j+n-1}{j-1} \\
    & = n\sum_{j = 0}^{n-1}(j-1)\binom{j+n-1}{j-1}+n\sum_{j = 0}^{n-1}\binom{j+n-1}{j-1}\\
    & = n\sum_{j = -1}^{n-2}j\binom{j+n}{j} + n\sum_{j = -1}^{n-2}\binom{(j-1)+n}{j-1}\\
    & = n\left[\sum_{j = -1}^n j\binom{j+(n-1)-1}{j} - (n-1)\binom{2n-1}{n-1}-n\binom{2n}{n}\right] \hspace{.25in}\text{(Parallel summation)}\\
    &\hspace{.5in}+ n\binom{2n-1}{n-2} 
     \ \ \ \\
    & = n\left[\frac{n(n+1)}{n+2}\binom{2(n+1)-1}{(n+1)-1} - (n-1)\binom{2n-1}{n-1}-n\binom{2n}{n}\right] \hspace{.25in}\text{(Proposition~\ref{eq:BinomialIdentity1})}\\
    &\hspace{.5in}+ n \binom{2n-1}{n-2}  \ \ \  
    \\
    & = \frac{n^3(n-1)}{(n+1)(n+2)}\binom{2n-1}{n-1}.\qedhere
\end{align*}

\end{proof}

The following lemma gives a recursive formula for the number of weakly increasing parking functions of length $n$ whose last entry is $j$ for $j \in [n]$.
\begin{lemma}\label{lem:monotonicityofWIPF}
Let $f_n(j)$ denote the number of weakly increasing parking functions $\alpha \in \NDPF_n$ of length $n \geq 2$ with $\alpha_n = j$, for $1 < j < n$.
Then,
  \[
  f_n(j) = f_n(j-1) + f_{n-1}(j).
    \]
    Moreover, $f_n(n) = f_n(n - 1)$.
\end{lemma}
\begin{proof}
    Let $A_n(j) \coloneqq \{\alpha \in \NDPF_n \colon \alpha_n = j\}$. Denote the set of prefixes of weakly increasing parking functions with last entry equal to $j$ by $B_n(j) \coloneqq \{(\alpha_1, \dots, \alpha_{n-1}) \colon \alpha \in A_n(j)\} \subset \NDPF_{n-1}$ and note that 
    $$
        B_n(j) = \bigsqcup_{i = 1}^j A_{n-1}(j).
    $$
    Since the elements in $A_n( j)$ are uniquely determined by their first $n-1$ coordinates, and since any prefix belonging to both $B_n(j)$ and $B_n(j - t)$ for $2 \leq t \leq j - 1$ also belongs to $B_n(j-1)$, we conclude 
    $$
        B_n(j) \setminus B_n(j-1) = A_{n-1}(j).
    $$
    Hence 
    $$
        f_n(j) - f_n(j-1) = |B_n(j)| - |B_n(j-1)| = f_{n-1}(j).
    $$
    The last part of the statement follows since $f_n(n) = |\NDPF_{n-1}| = f_n(n-1)$.
\end{proof}

It is known that the sequence $f_n(j)$ appears in the Catalan triangle read by rows \\(OEIS \seqnum{A009766}).

\begin{proposition}[Verifies Theorem~\ref{thm:ExpectedTimeOpenBoundaries}] \label{prop:SolutionToLinearSystemAppendix}
Let $0 \leq p \leq 1$ and $q \coloneqq 1 - p$. Let $i, s \in [n]$ be fixed integers and $0 \leq w_s \leq 1$ for every $s \leq i$. The equation
\be \label{eq:SolutionAppendix}
g_s= \begin{cases}
\displaystyle\frac{1}{3}(i^2 - s^2) & 
\mbox{ if } p = \frac{1}{2}\\
\displaystyle\frac{
    i\left(1 + \left(\frac{q}{p}\right)^i\right)
}{(p-q) \left(1 - \left(\frac{q}{p}\right)^i\right)} - \frac{s \left(1 + \left(\frac{q}{p}\right)^s\right)}{(p-q) \left(1 - \left(\frac{q}{p}\right)^s\right)}
&\mbox{ if } p\neq \frac{1}{2}
\end{cases}
\ee
is a solution to the linear system
\be\label{eq:LinearSystemAppendix}
\begin{cases}
g_0 =0 & \mbox{ if } s =0\\
g_i =0 & \mbox{ if } s =i\\
g_{s}w_s=pw_{s+1}\left(g_{s+1}\right)
+ qw_{s-1}\left(g_{s-1}\right)
+w_s \eqpd & \mbox{ if } 1\leq  s \leq i-1\eqpd
\end{cases}
  \ee 
\end{proposition}

\begin{proof}
    
The case of $g_0=g_i=0$ is easily shown to be satisfied. For $1 \leq s \leq i-1$, use that
\be\label{eq:ratioOfProb}
\frac{w_{s+1}}{w_s} = \frac{\left(p\cdot p^{s}-q \cdot q^{s}\right)}{p\left(p^{s}-q^{s}\right)} \ \ \mbox{ and } \ \ \frac{w_{s-1}}{w_s} = \frac{\left(q\cdot p^{s}-p\cdot q^{s}\right)}{q\left(p^{s}-q^{s}\right)} \eqcom \ee
as well as the fact that 
\be\label{eq:Substitution2}\frac{
    \left(1 + \left(\frac{q}{p}\right)^{s+1}\right)
}{ \left(1 - \left(\frac{q}{p}\right)^{s+1}\right)} =\frac{p \cdot p^{s}+q\cdot q^{s}}{p\cdot p^{s}-q\cdot q^{s}}\ \ \mbox{ and } \ \ \frac{\left(1 + \left(\frac{q}{p}\right)^{s-1}\right)
}{ \left(1 - \left(\frac{q}{p}\right)^{s-1}\right)} =\frac{q \cdot p^{s}+p\cdot q^{s}}{q\cdot p^{s}-p\cdot q^{s}} \eqpd \ee

Next we evaluate
the last case in \eqref{eq:LinearSystemAppendix}
using \eqref{eq:SolutionAppendix}. By \eqref{eq:ratioOfProb}, $p\frac{w_{s+1}}{w} + q\frac{w_{s-1}}{w} =1$ and so our only concern with respect to each term on the right hand side of \eqref{eq:SolutionAppendix}, is the $- \frac{s \left(1 + \left(\frac{q}{p}\right)^{s\pm1}\right)}{(p-q) \left(1 - \left(\frac{q}{p}\right)^{s\pm1}\right)}$ contribution. Thus, using \eqref{eq:ratioOfProb} and \eqref{eq:Substitution2}, we have that 
\begin{align*}
    g_s &=\frac{
    i\left(1 + \left(\frac{q}{p}\right)^i\right)
}{(p-q) \left(1 - \left(\frac{q}{p}\right)^i\right)}- \frac{1}{(p-q)(p^s-q^s)}\left[\frac{(s+1)\left(p\cdot p^{s}-q \cdot q^{s}\right) \left(p \cdot p^{s}+q\cdot q^{s}\right)}{p\cdot p^{s}-q\cdot q^{s}} \right.\\
&\hspace{.25in}+ 
\left.\frac{(s-1)\left(q\cdot p^{s}-p \cdot q^{s}\right) \left(q\cdot p^{s}+p\cdot q^{s}\right)}{q\cdot p^{s}-p\cdot q^{s}} \right]+1\\
&=\frac{
    i\left(1 + \left(\frac{q}{p}\right)^i\right)
}{(p-q) \left(1 - \left(\frac{q}{p}\right)^i\right)}- \frac{1}{(p-q)(p^s-q^s)}\left[(s+1)\left(p^{s+1}+ q^{s+1}\right) \right.\\
&\hspace{.25in}+ \left.(s-1)\left(qp^{s}+p q^{s}\right)\right]+1\\
&=\frac{
    i\left(1 + \left(\frac{q}{p}\right)^i\right)
}{(p-q) \left(1 - \left(\frac{q}{p}\right)^i\right)}- \frac{1}{(p-q)(p^s-q^s)}\left[(s\left(p^{s}+ q^{s}\right)+(p-q)(p^s-q^s)\right]+1\\
&=\frac{
    i\left(1 + \left(\frac{q}{p}\right)^i\right)
}{(p-q) \left(1 - \left(\frac{q}{p}\right)^i\right)}-\frac{s \left(1 + \left(\frac{q}{p}\right)^s\right)}{(p-q) \left(1 - \left(\frac{q}{p}\right)^s\right)}-1+1\eqcom
\end{align*}
as expected. 
A similar argument works when $p = \frac{1}{2}$.
\end{proof}

\section*{Acknowledgments}
The authors gratefully acknowledge that this collaboration began at the 2024 Graduate Research Workshop in Combinatorics, which is supported in part by NSF Grant DMS~--~1953445. The authors would like to sincerely thank Steve Butler for initial contributions and providing a key insight in Proposition~\ref{prop:NumberOfWipf}. The authors also thank Alice Contat, Ivailo Hartarsky, and Philippe Nadeau for fruitful discussion, and for introducing us to similar models in the literature.
Finally, we would like to thank Persi Diaconis for taking the time to speak with us about this project, for helpful suggestions, and for offering his unique perspective on the model. Parts of this project were also presented at the Oberwolfach Mini-Workshop: Mixing Times in the Kardar-Parisi-Zhang Universality Class where  useful discoveries were made, and helpful discussions were had.
\bibliographystyle{plain}
\bibliography{bibliography.bib}
\end{document}